\documentclass[reqno]{amsart}
\usepackage{amsmath}%
\usepackage{amsfonts}%
\usepackage{amssymb}%
\usepackage{graphicx}
\newtheorem{theorem}{Theorem}
\theoremstyle{plain}

\newtheorem{corollary}{Corollary}

\newtheorem{definition}{Definition}
\newtheorem{example}{Example}

\newtheorem{lemma}{Lemma}

\newtheorem{proposition}{Proposition}
\newtheorem{remark}{Remark}

\numberwithin{equation}{section}
\begin{document}
\begin{center}
\vspace*{1.3cm}

\textbf{SCALARIZATION IN VECTOR OPTIMIZATION BY FUNCTIONS WITH UNIFORM SUBLEVEL SETS}

\bigskip

by

\bigskip

PETRA WEIDNER\footnote{HAWK Hochschule f\"ur angewandte Wissenschaft und Kunst Hildesheim/\-Holzminden/ G\"ottingen University of Applied Sciences and Arts, Faculty of Natural Sciences and Technology,
D-37085 G\"ottingen, Germany, {petra.weidner@hawk-hhg.de}.}

\bigskip
\bigskip
Research Report \\ 
Version 3 from December 05, 2017\\
Improved, Extended Version of Version 1 from June 18, 2016

\end{center}

\bigskip
\bigskip

\noindent{\small {\textbf{Abstract:}}
In this paper, vector optimization is considered in the framework of decision making and optimization in general spaces.
Interdependencies between domination structures in decision making and domination sets in vector optimization are given.
We prove some basic properties of efficient and of weakly efficient points in vector optimization.
Sufficient conditions for solutions to vector optimization problems are shown using minimal solutions of functionals.
We focus on the scalarization by functions with uniform sublevel sets, which also delivers necessary conditions for efficiency and weak efficiency. The functions with uniform sublevel sets may be, e.g., continuous or even Lipschitz continuous, convex, strictly quasiconcave or sublinear. They can coincide with an order unit norm on a subset of the space.
}

\bigskip

\noindent{\small {\textbf{Keywords:}} 
optimization in general spaces; vector optimization; decision making; scalarization
}

\bigskip

\noindent{\small {\textbf{Mathematics Subject Classification (2010): }
90C48, 90C29, 90B50}}

\section{Introduction}

In this report, we consider vector optimization problems as optimization problems in general spaces. The way to approach these problems is motivated by decision making in Section \ref{s-decision}, where the relationship between decision making and optimal elements w.r.t. domination sets is investigated. In Section \ref{s-basics-vo}, we define the vector optimization problem and give some basic results for the efficient point set and the weakly efficient point set including localization and existence. Sufficient conditions for efficiency and weak efficiency are given using minimal solutions of scalar-valued functions. Scalarization results for the efficient point set and the weakly efficient point set using functionals with uniform sublevel sets are proved in Section \ref{s-scal-uni}. These results include necessary conditions for efficiency and weak efficiency. Since we need several statements for functions with uniform sublevels sets and since properties of these functions are essential within their application in vector optimization procedures, the corresponding results from \cite{Wei16a} are summarized in Section \ref{s-uni}. These results also connect the functionals to norms. As a consequence, scalarization by norms is investigated in Section \ref{s-scal-norm}.

The report contains and extends results from \cite{Wei83}, \cite{Wei85}, \cite{GerWei90} and \cite{Wei90} on the basis of
\cite{Wei16a}. 

The functionals we consider can be real-valued or also attain values $-\infty$ or $+\infty$, but we
will use the symbolic function value $\nu $ (instead of the value $+\infty$ in convex analysis) when extending a functional to the entire space or in points where a function is not feasible otherwise. Thus, our approach differs from the classical one in convex analysis in these cases since the functions we use are of interest in minimization problems as well as in maximization problems. Consequently, we consider functions which can attain values in $\overline{\mathbb{R}}_{\nu }:=\overline{\mathbb{R}}\cup\{\nu \}$, where $\overline{\mathbb{R}}:=\mathbb{R}\cup\{-\infty, +\infty\}$. Details of functions with values in $\overline{\mathbb{R}}_{\nu }$ are explained in \cite{Wei15}. 

From now on, $\mathbb{R}$ will denote the set of real numbers.
We define $\mathbb{R}_{+}:=\{x\in\mathbb{R}\mid x\geq 0\}$, $\mathbb{R}_{>}:=\{x\in\mathbb{R}\mid x > 0\}$ and
$\mathbb{R}_{+}^2:=\{(x_1,x_2)\in\mathbb{R}^2\mid x_1\geq 0, x_2\geq 0\}$.
Linear spaces will always be assumed to be real vector spaces. 
Given two sets $A\subseteq \mathbb{R}$, $B\subseteq Y$ and some vector $k$ in $Y$ in a linear space $Y$, we will use the notation $A\; B:=A \cdot B:=\{a \cdot b\mid a \in A , \; b\in B\}$ and $A\; k:=A \cdot k:=A \cdot \{ k\}$.
A set $C$ in a linear space $Y$ is a cone if $\mathbb{R}_{+} C\subseteq C.$ The cone $C$ is called non-trivial if $C\not=\emptyset$, $C\not=\{0\}$ and $C\not= Y$ hold. For a subset $A$ of some linear space $Y$, 
$\operatorname*{core}A$ will denote the algebraic interior of $A$ and $0^+A:=\{u\in Y  \mid  \ \forall a\in A \; \forall t\in \mathbb{R}_{+}\colon a+tu\in A\}$ the recession cone of $A$.
In a topological space $Y$, $\operatorname*{cl}A$, $\operatorname*{int}A$ and $\operatorname*{bd}A$ denote the closure, the interior and the boundary, respectively, of a subset $A$.
For a functional $\varphi$ defined on some space $Y$ and attaining values in $\overline{\mathbb{R}}_{\nu }$, we will denote the epigraph of $\varphi$ by   
$\operatorname*{epi}\varphi$ and the effective domain of $\varphi$ by $\operatorname*{dom}\varphi$. $\mathcal{P}(Y)$ stands for the power set of $Y$.

Beside the properties of functions defined in \cite{Wei15}, we will need the following ones:
\begin{definition}\label{d-mon}
Let $Y$ be a linear space, $B\subseteq Y$ and $\varphi: Y \to \overline{\mathbb{R}}_{\nu } $. \\
$\varphi$ is said to be
\begin{itemize}
\item[(a)]
$B$-monotone on $F\subseteq \operatorname*{dom}\varphi$
if $y^1,y^2 \in F$ and $y^{2}-y^{1}\in B$ imply $\varphi
(y^{1})\le \varphi (y^{2})$,
\item[(b)] strictly $B$-monotone on $F\subseteq \operatorname*{dom}\varphi$ 
if $y^1,y^2 \in F$ and $y^{2}-y^{1}\in B\setminus
\{0\}$ imply $\varphi (y^{1})<\varphi (y^{2})$,
\item[(c)] $B$-monotone or strictly $B$-monotone if it is $B$-monotone or strictly $B$-monotone, respectively, on $\operatorname*{dom}\varphi$,
\item[(d)] quasiconvex if $\operatorname*{dom}\varphi$ is convex and
\[\varphi (\lambda y^1 + (1-\lambda ) y^2) \le \operatorname*{max}(\varphi (y^1),\varphi (y^2))\]
for all $y^1, y^2 \in \operatorname*{dom}\varphi$ and $\lambda \in (0,1)$,
\item[(e)] strictly quasiconvex if $\operatorname*{dom}\varphi$ is convex and
\[\varphi (\lambda y^1 + (1-\lambda ) y^2) < \operatorname*{max}(\varphi (y^1),\varphi (y^2))\]
for all $y^1, y^2 \in \operatorname*{dom}\varphi$ with $y^1\not= y^2$ and $\lambda \in (0,1)$,
\item[(f)] strictly quasiconcave if $-\varphi$ is strictly quasiconvex. 
\end{itemize}
\end{definition}
\smallskip

\section{Functions with Uniform Sublevel Sets}\label{s-uni}

In this section, we summarize a part of the results for functionals with uniform sublevel sets which we have proved in \cite{Wei16a}. These functions will be used for scalarization in Section \ref{s-scal-uni}. Hence, their properties are of special importance for procedures which determine solutions to vector optimization problems. Moreover, we will give an existence result for efficient points in vector optimization, which is proved by using the functions introduced in this section, in Section \ref{s-basics-vo}.

\quad

\begin{tabular}{ll}
Assuming $(H1_{a-H,k})$: & $Y$ is a topological vector space, $a\in Y$,\\
& $H$ is a closed proper subset of $Y$ and $k\in 0^+H\setminus \{0\}$, 
\end{tabular}

we define $\varphi _{a-H,k}:Y\rightarrow \overline{{\mathbb{R}}}_{\nu }$ by $\varphi_{a-H,k} (y) := \inf \{t\in
{\mathbb{R}} \mid y\in a-H+tk\}$.\\
For $\varphi _{0-H,k}$, we will simply write $\varphi _{-H,k}$.

We have to keep in mind that we use the following terms and definitions according to \cite{Wei15}:
\begin{enumerate}
\item $\operatorname*{inf}\emptyset =\nu\not\in\overline{\mathbb{R}},$
\item $\operatorname*{dom}\varphi _{a-H,k}=\{y\in Y\mid \varphi _{a-H,k}(y)\in\mathbb{R}\cup (-\infty )\}$ is the (effective) domain of $\varphi _{a-H,k},$
\item $\varphi _{a-H,k}$ is said to be finite-valued if $\varphi _{a-H,k}(y)\in\mathbb{R} \mbox{ for all } y\in Y,$ 
\item $\varphi _{a-H,k}$ is said to be finite-valued on $F\subseteq Y$ if $\varphi _{a-H,k}(y)\in\mathbb{R} \mbox{ for all } y\in F,$ 
\item $\varphi _{a-H,k}$ is said to be proper if $\operatorname*{dom}\varphi _{a-H,k}\not= \emptyset$ and $\varphi _{a-H,k}$ is finite-valued on $\operatorname*{dom}\varphi _{a-H,k}.$
\end{enumerate}

\begin{lemma}\label{t251ua}
%\label{t251}\label{k-core}\label{t251M}\label{varphi_cx_proper}
Assume $(H1_{a-H,k})$.\\
Then $\varphi _{a-H,k}$ is lower semicontinuous on $\operatorname*{dom}\varphi _{a-H,k}=a-H+{\mathbb{R}}k \not= \emptyset$.\\
Moreover,
\begin{itemize}
\item[(a)] $\{y\in Y\mid \varphi_{a-H,k}(y)\leq t\}  =  a-H+t k \mbox{ for all } t \in {\mathbb{R}}$.
\item[(b)] $\varphi _{a-H,k}$ is finite-valued on $\operatorname*{dom}\varphi _{a-H,k}\setminus (a-H)$. 
\item[(c)] If $k\in \operatorname*{core}0^+H$, then $\varphi _{a-H, k}$ is finite-valued and\\
$\{y\in Y\mid \varphi_{a-H,k}(y)< t\}=\operatorname*{core}A+tk  \mbox{ for all } t \in \mathbb{R}$.
\item[(d)] If $k\in 0^+H\cap (-0^+H)$, then $\varphi_{a-H,k}$ does not attain any real value.
\item[(e)] $\varphi _{a-H,k}(y)=-\infty \iff y+\mathbb{R}k\subseteq a-H$.
\item[(f)] $\varphi _{a-H,k} $ is positively homogeneous $\iff$ $a-H$ is a cone.
\item[(g)] $\varphi _{a-H,k} $ is sublinear $\iff$ $a-H$ is a convex cone.
\item[(h)] If $H+D\subseteq H$ for $D\subset Y$, then $\varphi _{a-H,k}$ is D-monotone.
\item[(i)] If $\varphi _{a-H,k}$ is finite-valued on $F\subseteq Y$, $D\subset Y$ and
$H+(D\setminus \{0\}) \subseteq \operatorname*{core} H$, then $\varphi _{a-H,k} $ is strictly $D$-monotone on $F$.
\end{itemize}
\end{lemma}

\begin{lemma}\label{prop-funcII}
Assume $(H2_{a-H,k})$:\quad $(H1_{a-H,k})$ holds and $H+\mathbb{R}_{>} k\subseteq \operatorname*{int}H$.\\
Then $\varphi _{a-H,k}$ is continuous on $\operatorname*{dom}\varphi _{a-H,k}$,
\begin{eqnarray*}
\{y\in Y\mid \varphi_{a-H,k}(y)< t\} &= & a-\operatorname*{int}H+t k\mbox{ for all } t \in {\mathbb{R},} \\
\{y\in Y\mid \varphi_{a-H,k}(y)= t\} & = & a-\operatorname*{bd}H+t k\mbox{ for all } t \in {\mathbb{R}}. 
\end{eqnarray*} 
Moreover,
\begin{itemize}
\item[(a)] $\varphi _{a-H,k}$ is finite-valued on $\operatorname*{dom}\varphi _{a-H,k}\setminus (a-\operatorname*{int}H)$. 
\item[(b)] If $\varphi _{a-H,k} $ is proper, then:\\
$\varphi _{a-H,k} $ is strictly quasiconvex $\iff$ $H$ is a strictly convex set.
\item[(c)] If $\varphi _{a-H,k}$ is finite-valued, then:\\
$\varphi _{a-H,k}$ is concave $\iff Y\setminus \operatorname*{int}H$ is convex,\\
$\varphi _{a-H,k}$ is strictly quasiconcave $\iff Y\setminus\operatorname*{int}H$ is a strictly convex set.
\end{itemize}
\end{lemma}

\begin{lemma}\label{-phi_Ak}
Assume $(H2_{a-H,k})$. Then $(H2_{Y\setminus(a-\operatorname*{int}H),-k})$ holds and
\begin{itemize}
\item[(a)] $a-\operatorname*{bd}H +\mathbb{R} k$ is the subset of $Y$ on which $\varphi _{a-H,k}$ is finite-valued as well as the subset of $Y$ on which 
$\varphi _{Y\setminus (a-\operatorname*{int}H),-k}$ is finite-valued,
\item[(b)] $\operatorname*{dom}\varphi _{a-H,k}\cap\operatorname*{dom}\varphi _{Y\setminus (a-\operatorname*{int}H),-k}=a-\operatorname*{bd}H +\mathbb{R} k$,
\item[(c)] $\varphi _{a-H,k}(y)=-\varphi _{Y\setminus (a-\operatorname*{int}H),-k}(y)\mbox{ for all } y\in a-\operatorname*{bd}H +\mathbb{R} k.$
\end{itemize}
\end{lemma}

\begin{lemma}\label{prop-finite_ua}
%\label{prop-finite}\label{th-Lip-varphi}
Assume $(H1_{a-H,k})$ holds and $k\in \operatorname*{int}0^+H$. \\
Then $(H2_{a-H,k})$ holds, and $\varphi_{a-H,k}$ is continuous and finite-valued.\\
If $Y$ is a Banach space, then $\varphi_{a-H,k}$ is Lipschitz continuous.
\end{lemma}

\begin{lemma}\label{Fact4}
Assume that $Y$ is a topological vector space, $a\in Y$, $H$ is a proper closed convex subset of $\;Y$ and $k\in 0^+H\setminus (-0^+H)$.\\
Then $(H1_{a-H,k})$ is fulfilled and the following statements are valid.
\begin{itemize}
\item[(a)] $\varphi _{a-H,k}$ is convex, proper and lower semicontinuous on $\operatorname*{dom}\varphi _{a-H,k}$.
\item[(b)] If $\operatorname*{int}H\not=\emptyset$, then $\varphi _{a-H,k}$ is continuous on $\operatorname*{int}\operatorname*{dom}\varphi _{a-H,k}$.
\item[(c)] If $\;Y$ is a Banach space, then $\varphi_{a-H,k}$ is locally Lipschitz on $\operatorname*{int}\operatorname*{dom}\varphi_{a-H,k}$.
\end{itemize}
\end{lemma}

\begin{lemma}\label{c251a-C}
Assume that $Y$ is a topological vector space, $H\subset Y$ is a non-trivial closed convex cone, $k\in H\setminus\{0\}$ and $a\in Y$.
Then $(H1_{a-H,k})$ holds, and $\varphi_{a-H,k}$ is a convex $H$-monotone functional, which is lower semicontinuous on its domain.
\begin{itemize}
\item[(a)] If $k\in H\cap (-H)$, then $\operatorname*{dom}\varphi_{a-H,k}=a-H$ and $\varphi_{a-H,k}$ does not attain any real value.
\item[(b)] If $k\in H \setminus (-H)$, then $\varphi_{a-H,k}$ is proper and strictly $(\operatorname*{core} H)$-monotone.
\item[(c)] If $k\in \operatorname*{core} H$, then $\varphi_{a-H,k}$ is finite-valued.
\item[(d)] If $k\in \operatorname*{int} H$, then 
$(H2_{a-H,k})$ holds and $\varphi_{a-H,k}$ is continuous, finite-valued and strictly $(\operatorname*{int} H)$-monotone.
\item[(e)] $\varphi_{a-H,k}$ is subadditive $\iff a\in -H$.
\item[(f)] $\varphi_{a-H,k}$ is sublinear $\iff a\in H\cap(-H)$.
\end{itemize}
\end{lemma}

For cones $H$, functionals $\varphi_{a-H,k}$ are closely related to norms.

\begin{lemma}\label{p-ordint-varphi}
Suppose that $Y$ is a topological vector space, $H \subset Y$ a non-trivial closed convex pointed cone with $k\in\operatorname*{core}H$, $a\in Y$. Denote by $\|\cdot \|_{H,k}$ the norm which is given as the Minkowski functional of the order interval $[-k,k]_H$. Then
\[ 
\| y-a\| _{H,k}=\varphi_{a-H,k}(y)\mbox{ for all } y\in a+H.
\]
\end{lemma}

Let us now investigate the influence of the choice of $k$ on the values of $\varphi _{a-H, k}$.
\begin{lemma}\label{t-scale}
Assume $(H1_{a-H,k})$, and consider some arbitrary $\lambda\in\mathbb{R}_{>}$.\\
Then $(H1_{a-H,\lambda k})$ holds, $\operatorname*{dom}\varphi _{a-H,\lambda k}=\operatorname*{dom}\varphi _{a-H,k}$ and
\begin{displaymath}
\varphi_{a-H,\lambda k}(y)=\frac{1}{\lambda} \varphi_{a-H,k}(y) \mbox{ for all } y\in Y.
\end{displaymath}
$\varphi_{a-H,\lambda k}$ is 
proper, finite-valued, continuous, lower semicontinuous, upper semicontinuous, convex, concave, strictly quasiconvex, subadditive, superadditive, affine, linear, sublinear, positively homogeneous, odd or homogeneous if and only if $\varphi_{a-H,k}$ has the same property. 
For $B\subset Y$, the functional $\varphi_{a-H,\lambda k}$ is 
$B$-monotone or strictly $B$-monotone if and only
if $\varphi_{a-H,k}$ has the same property.
\end{lemma}

The lemma underlines that replacing $k$ by another vector into the same direction just scales the functional. Consequently, $\varphi _{A, k}$ and $\varphi _{A, \lambda k}$, $\lambda >0$, take optimal values on some set $F\subset Y$ at the same elements of $F$. Hence, it is sufficient to consider only one vector $k$ per direction in optimization problems, e.g., to restrict $k$ to unit vectors if $Y$ is a normed space.

The functionals $\varphi _{a-H,k}$ can be described by $\varphi _{-H,k}$ and thus by a sublinear function if $H$ is a convex cone.
\begin{lemma}\label{A-shift}
Assume $(H1_{-H,k})$, and consider some arbitrary $a\in Y$.\\
Then $(H1_{a-H,k})$ is satisfied, $\operatorname*{dom}\varphi _{a-H,k}=a+\operatorname*{dom}\varphi _{-H,k}$ and
\begin{displaymath}
\varphi_{a-H,k}(y)=\varphi_{-H,k}(y-a) \mbox{ for all } y\in Y.
\end{displaymath}
$\varphi_{a-H,k}$ is 
proper, finite-valued, continuous, lower semicontinuous, upper semicontinuous, convex, concave, strictly quasiconvex or affine if and only if $\varphi_{-H,k}$ has the same property.\\
For $B\subset Y$, $\varphi_{a-H,k}$ is
$B$-monotone or strictly $B$-monotone if and only if $\varphi_{-H,k}$ has the same property.
\end{lemma}
\smallskip

\section{Decision Making and Vector Optimization}\label{s-decision}

Consider the following general decision problem:\\
A decision maker (DM) wants to make a decision by choosing an element from a set $S$ of feasible decisions, where the outcomes of the decisions are given by some function $f:S\to Y$, $Y$ being some arbitrary set.

What is a best decision depends on the DM's preferences in the set $F:=f(S)$ of outcomes.  
A relation $\succ$ is called a \textbf{strict preference relation} on $Y$ if $\forall y^1, y^2\in Y$:
$$y^1\succ y^2\iff y^1\mbox{ is strictly preferred to }y^2.$$
$\succ$ is said to be a \textbf{weak preference relation} on $Y$ if $\forall y^1, y^2\in Y$:
$$y^1\succ y^2\iff y^2\mbox{ is not preferred to }y^1.$$

Let $\succ$ denote the DM's strict or weak preference relation on $Y$. Then the set of decision outcomes which are optimal for the DM is just $\operatorname*{Min}(F,\succ ):=\{y^0\in F\mid \forall y\in F: (y\succ y^0\Rightarrow y^0\succ y)\}$, i.e., optimal decisions are the elements of $\operatorname*{Min}_f(S,\succ ):=\{x\in S\mid f(x)\in \operatorname*{Min}(f(S),\succ )\}$.

Note that $\succ$ consists of the preferences the DM is aware of. This relation is refined during the decision process, but fixed in each single step of the decision process, where information about $\operatorname*{Min}(F,\succ )$ should support the DM in formulating further preferences. In the final phase of the decision process, the DM chooses one decision, but in the previous phases $\operatorname*{Min}(F,\succ )$ contains more than one element.

\begin{definition}
Suppose $\succ$ to be a binary relation on the linear space $Y$.\\
$d\in Y$ is said to be a \textbf{domination factor} of $y\in Y$ if $y\succ y+d$.
We define the \textbf{domination structure} of $\succ$ by $D_{\succ }:Y\to {\mathcal{P}}(Y)$ with $D_{\succ }(y):=\{d\in Y\mid y\succ y+d\}$ for each $y\in Y$.
If there exists some set $D\subseteq Y$ with $D_{\succ }(y)=D$ for all $y\in Y$, then $D$ is called the \textbf{domination set} of $\succ$.
\end{definition}

The definition implies:

\begin{proposition}
Suppose $\succ$ to be a binary relation on the linear space $Y$, $D\subseteq Y$.\\
$D$ is a domination set of $\succ$ if and only if:
$$\forall y^1,y^2\in Y:\;\; (y^1\succ y^2\iff y^2\in y^1+D).$$
There exists a domination set of $\succ$ if and only if:
\begin{equation}\label{domset-add}
\forall y^1,y^2,y\in Y:\;\; (y^1\succ y^2\implies (y^1+y)\succ (y^2+y)).
\end{equation}
If $D$ is a domination set of $\succ$, we have:
\begin{itemize}
\item[(a)] $\succ$ is reflexive $\iff 0\in D$.
\item[(b)] $\succ$ is asymmetric $\iff D\cap (-D)=\emptyset$.
\item[(c)] $\succ$ is antisymmetric $\iff D\cap (-D)=\{0\}$.
\item[(d)] $\succ$ is transitive $\iff D+D\subseteq D$.
\item[(e)] $\succ$ fulfills the condition
\begin{equation}
\forall\, y^{1},y^{2}\in Y \;\; \forall\, \lambda \in {\mathbb{R}_>}:
\;\; y^{1} \succ y^{2} \implies (\lambda y^{1})\succ (\lambda y^{2}), \label{f217-y} 
\end{equation}
if and only if $D\cup\{0\}$ is a cone.
\item[(f)] $\succ$ is a transitive relation which fulfills condition (\ref{f217-y}) if and only if $D\cup\{0\}$ is a convex cone.
\item[(g)] $\succ$ is a partial order which fulfills condition (\ref{f217-y}) if and only if $D$ is a pointed convex cone.
\end{itemize}
\end{proposition}

\begin{example}
Not each preference relation fulfills the conditions (\ref{domset-add}) and (\ref{f217-y}). Consider $y$ to be the number of tea spoons full of sugar that a person puts into his coffee. He could prefer $2$ to $1$, but possibly not $2+2$ to $1+2$ or also not prefer $2\times 2$ to $2\times 1$.
\end{example}

We now introduce optimal elements w.r.t. sets as a tool for finding optimal elements w.r.t. relations. 

\begin{definition}
Suppose $Y$ to be a linear space and $F, D\subseteq Y$.
An element $y^{0}\in F$ is called an \textbf{efficient element} of $F$
w.r.t. $D$ if 
\[
F\cap(y^0-D)\subseteq \{y^0\}.
\]
We denote the set of efficient elements of $F$ w.r.t. $D$ by $\operatorname*{Eff}(F,D).$
\end{definition}

We get \cite[p.51]{Wei85}:

\begin{proposition}\label{p-domstruct-D}
Suppose $\succ$ to be a (not necessarily strict) preference relation on the linear space $Y$ with domination structure $D_{\succ }$, $D\subseteq Y$.
\begin{itemize}
\item[(a)] If $D_{\succ }(y)=D\mbox{ for all } y\in F$, then $\operatorname*{Min}(F,\succ)=\operatorname*{Eff}(F,D\setminus (-D))$.\\
If $\succ$ is asymmetric or antisymmetric, then $\operatorname*{Min}(F,\succ)=\operatorname*{Eff}(F,D)$. 
\item[(b)] If $D_{\succ }(y)\subseteq D\mbox{ for all } y\in F$, then $\operatorname*{Eff}(F,D)\subseteq\operatorname*{Min}(F,\succ)$.
\item[(c)] If $\succ$ is asymmetric or antisymmetric and $D\subseteq D_{\succ }(y)\mbox{ for all } y\in F$, then\\
$\operatorname*{Min}(F,\succ)\subseteq \operatorname*{Eff}(F,D)$.
\end{itemize}
\end{proposition}

\begin{proof}
\begin{itemize}
\item[]
\item[(a)] Consider some $y^0\in F$.
\begin{eqnarray*}
y^0\notin \operatorname*{Min}(F,\succ) & \Leftrightarrow & \exists y\in F:\;y\succ y^0,\mbox{ but } \neg (y^0\succ y),\\
& \Leftrightarrow & \exists y\in F:\;y^0\in y+D,\mbox{ but } y\notin y^0+D,\\
& \Leftrightarrow & \exists y\in F:\;y^0\in y+(D\setminus (-D))\\
& \Leftrightarrow & y^0\notin \operatorname*{Eff}(F,D\setminus (-D)).
\end{eqnarray*}
If $\succ$ is asymmetric or antisymmetric, then $D\cap(-D)\subseteq\{0\}$. $\Rightarrow D\setminus (-D)=D$ or $D\setminus (-D)=D\setminus\{0\}$. $\Rightarrow \operatorname*{Eff}(F,D\setminus (-D))=\operatorname*{Eff}(F,D)$.
\item[(b)] Consider some $y^0\in F\setminus \operatorname*{Min}(F,\succ)$.
$\Rightarrow \exists y\in F:\;y\succ y^0,\mbox{ but } \neg (y^0\succ y)$.
$\Rightarrow \exists y\in F\setminus\{y^0\}:\; y^0\in y+D_{\succ }(y)$.
$\Rightarrow \exists y\in F\setminus\{y^0\}:\; y\in y^0-D_{\succ }(y)\subseteq y^0-D$. 
$\Rightarrow y^0\notin \operatorname*{Eff}(F,D)$.
\item[(c)] Consider some $y^0\in F\setminus \operatorname*{Eff}(F,D)$.
$\Rightarrow \exists y\in F\setminus\{y^0\}:\; y\in y^0-D\subseteq y^0-D_{\succ }(y)$. 
$\Rightarrow \exists y\in F\setminus\{y^0\}:\; y^0\in y+D_{\succ }(y)$.
$\Rightarrow \exists y\in F\setminus\{y^0\}:\;y\succ y^0$.
$\Rightarrow y^0\notin \operatorname*{Min}(F,\succ)$ since $\succ$ is asymmetric or antisymmetric.
\end{itemize}
\end{proof}

\begin{remark}
In \cite{Wei83} and \cite{Wei85}, optimal elements of sets w.r.t. relations and sets were investigated under the general assumptions given here. There exist earlier papers which study optima w.r.t. quasi orders, e.g. \cite{Wall45} and \cite{Ward54}, or optimal elements w.r.t. ordering cones, e.g. \cite{hurw58} and \cite{yu73dom}. The concept of domination structures goes back to Yu \cite{yu73dom}. Domination factors according to the above definition were introduced in \cite{Berg76},
where minimal elements w.r.t. convex sets $D$ with $0\in D\setminus\operatorname*{int}D$ were investigated in $\mathbb{R}^\ell$.   
\end{remark}

The domination factors refer to elements which are dominated. Of course, a structure could also be built by dominating elements. Such a structure was studied by Chen \cite{Chen92} and later in the books by himself et al. \cite{CheHuaYan05} for the case that the structure consists of convex cones or of convex sets which contain zero in their boundary.

\begin{definition}
Suppose $\succ$ to be a binary relation on the linear space $Y$.\\
$\tilde{d}\in Y$ is said to be a \textbf{pre-domination factor} of $y\in Y$ if $\;y-\tilde{d}\succ y$.
We define the \textbf{pre-domination structure} of $\succ$ by $\tilde{D}_{\succ } : Y\to {\mathcal{P}}(Y)$ with $\tilde{D}_{\succ }(y):=\{\tilde{d}\in Y\mid y-\tilde{d}\succ y\}$ for each $y\in Y$.
\end{definition}

A pre-domination structure is constant on the entire space if and only if the domination structure of the same relation is constant on the whole space.

\begin{proposition}
Suppose $\succ$ to be a (not necessarily strict) preference relation on the linear space $Y$ with pre-domination structure $\tilde{D}_{\succ }$.
There exists a domination set $D\subseteq Y$ of $\succ$ if and only if $\tilde{D}_{\succ }(y)=D\mbox{ for all } y\in Y$.
\end{proposition} 

\begin{proof}
Let $D_{\succ }$ denote the domination structure of $\succ$.\\
$D_{\succ }(y)=D$ holds for all $y\in Y$ if and only if:\\
$\forall y\in Y:\;(y\succ y+d\Leftrightarrow d\in D)$, i.e., if and only if\\
$\forall y\in Y:\;(y-d\succ y\Leftrightarrow d\in D)$, which is equivalent to $\tilde{D}_{\succ }(y)=D\mbox{ for all } y\in Y$.
\end{proof}

The pre-domination structure may consist of convex sets when this is not the case for the domination structure.

\begin{example}\label{ex-predom}
Define on $Y=\mathbb{R}^2$ the relation $\succ$ by: $y^1\succ y^2\Leftrightarrow \| y^1\| _2\leq \| y^2\| _2$.
The pre-domination structure, but not the domination structure,  consists of convex sets. 
\end{example}

Analogously to Proposition \ref{p-domstruct-D}, the following relationships between pre-domination structures and minima w.r.t. sets hold.

\begin{proposition}
Suppose $\succ$ to be a (not necessarily strict) preference relation on the linear space $Y$ with pre-domination structure $\tilde{D}_{\succ }$, $D\subseteq Y$.
\begin{itemize}
\item[(a)] If $\tilde{D}_{\succ }(y)=D\mbox{ for all } y\in F$, then $\operatorname*{Min}(F,\succ)=\operatorname*{Eff}(F,D\setminus (-D))$.\\
If $\succ$ is asymmetric or antisymmetric, then $\operatorname*{Min}(F,\succ)=\operatorname*{Eff}(F,D)$. 
\item[(b)] If $\tilde{D}_{\succ }(y)\subseteq D\mbox{ for all } y\in F$, then $\operatorname*{Eff}(F,D)\subseteq\operatorname*{Min}(F,\succ)$.
\item[(c)] If $\succ$ is asymmetric or antisymmetric and $D\subseteq \tilde{D}_{\succ }(y)\mbox{ for all } y\in F$, then\\
$\operatorname*{Min}(F,\succ)\subseteq \operatorname*{Eff}(F,D)$.
\end{itemize}
\end{proposition}

\begin{proof}
\begin{itemize}
\item[]
\item[(a)] Consider some $y^0\in F$.
\begin{eqnarray*}
y^0\notin \operatorname*{Min}(F,\succ) & \Leftrightarrow & \exists y\in F:\;y\succ y^0,\mbox{ but } \neg (y^0\succ y),\\
& \Leftrightarrow & \exists y\in F:\;y\in y^0-D,\mbox{ but } y^0\notin y-D,\\
& \Leftrightarrow & \exists y\in F:\;y^0\in y+(D\setminus (-D))\\
& \Leftrightarrow & y^0\notin \operatorname*{Eff}(F,D\setminus (-D)).
\end{eqnarray*}
If $\succ$ is asymmetric or antisymmetric, then $D\cap(-D)\subseteq\{0\}$. $\Rightarrow D\setminus (-D)=D\setminus\{0\}$. $\Rightarrow \operatorname*{Eff}(F,D\setminus (-D))=\operatorname*{Eff}(F,D)$.
\item[(b)] Consider some $y^0\in F\setminus \operatorname*{Min}(F,\succ)$.
$\Rightarrow \exists y\in F:\;y\succ y^0,\mbox{ but } \neg (y^0\succ y)$.
$\Rightarrow \exists y\in F\setminus\{y^0\}:\; y\in y^0-\tilde{D}_{\succ }(y^0)\subseteq y^0-D$.
$\Rightarrow y^0\notin \operatorname*{Eff}(F,D)$.
\item[(c)] Consider some $y^0\in F\setminus \operatorname*{Eff}(F,D)$.
$\Rightarrow \exists y\in F\setminus\{y^0\}:\; y\in y^0-D\subseteq y^0-\tilde{D}_{\succ }(y^0)$. 
$\Rightarrow \exists y\in F\setminus\{y^0\}:\;y\succ y^0$.
$\Rightarrow y^0\notin \operatorname*{Min}(F,\succ)$ since $\succ$ is asymmetric or antisymmetric.
\end{itemize}
\end{proof}

Since the domination structure as well as the pre-domination structure completely characterize the binary relation, the minimal point set w.r.t. the relation can be described via the domination structure or via the pre-domination structure. 

\begin{lemma}\label{l-minrel-mindom}
Suppose $\succ$ to be a binary relation on the linear space $Y$ with domination structure $D_{\succ }$ and pre-domination structure $\tilde{D}_{\succ }$, $F\subseteq Y$.
\begin{eqnarray*}
\operatorname*{Min}(F,\succ) &= &\{y^0\in F\mid \forall y\in F:\;(y^0\in y+D_{\succ }(y)\Rightarrow y\in y^0+D_{\succ }(y^0))\}\\
& = & \{y^0\in F\mid \forall y\in F:\;(y\in y^0-\tilde{D}_{\succ }(y^0)\Rightarrow y^0\in y-\tilde{D}_{\succ }(y))\}\\
& = & \{y^0\in F\mid \forall y\in F:\;(y^0\in y+\tilde{D}_{\succ }(y^0)\Rightarrow y\in y^0+\tilde{D}_{\succ }(y))\}
\end{eqnarray*}
If $\succ$ is asymmetric or antisymmetric, we get
\begin{eqnarray*}
\operatorname*{Min}(F,\succ) &= &\{y^0\in F\mid \not\exists y\in F\setminus\{y^0\}:\;y^0\in y+D_{\succ }(y)\}\\
& = & \{y^0\in F\mid \not\exists y\in F\setminus\{y^0\}:\;y\in y^0-\tilde{D}_{\succ }(y^0)\}\\
& = & \{y^0\in F\mid \not\exists y\in F\setminus\{y^0\}:\;y^0\in y+\tilde{D}_{\succ }(y^0)\}
\end{eqnarray*}
\end{lemma}

In the case that the domination structure can be described by a domination set, the decision problem becomes a vector optimization problem, which we will study in the next sections.
\smallskip

\section{Basic Properties of Efficient and Weakly Efficient Elements in Vector Optimization}\label{s-basics-vo}

In this section, we will define the vector optimization problem and prove some basic properties of its solutions, including existence results. We will show in which way minimal solutions of scalar-valued functions can deliver solutions to the vector optimization problem.

The \textbf{vector optimization problem} is given by a function $f:S\rightarrow Y$, mapping a nonempty set $S$ into a linear space $Y$, and a set $D\subset Y$ which defines the solution concept. A solution of the vector optimization problem is each $s\in S$ with $f(s)\in \operatorname*{Eff}(f(S),D)$. 

Hence, we are interested in the efficient elements of $F:=f(S)$ w.r.t. $D$. One can imagine that for each $y^0\in F$ the set of elements in $F$ which is preferred to $y^0$ is just $F\cap (y^0-(D\setminus\{0\}))$.
We will call $D$ the \textbf{domination set} of the vector optimization problem.

\begin{remark}
Weidner (\cite{Wei83}, \cite{Wei85}, \cite{Wei90}) studied vector optimization problems under such general assumptions motivated by decision theory. Here, we only refer to a part of those results. If $D$ is an ordering cone in $Y$, $\operatorname*{Eff}(F,D)$ is the set of elements of $F$ which are minimal w.r.t. the cone order $\leq_D$. In the literature, vector optimization problems are usually defined with domination sets which are ordering cones.
\end{remark}

If $Y$ is a topological vector space, it turns out that, in general, it is easier to determine solutions to vector optimization problems w.r.t. open domination sets.
\begin{definition}
Suppose $Y$ to be a topological vector space and $F, D\subseteq Y$.
We define $\operatorname*{WEff}(F,D):=\operatorname*{Eff}(F,\operatorname*{int}D)$ as the set of 
\textbf{weakly efficient elements} of $F$ w.r.t. $D$.
\end{definition}

We will first show some basic properties of efficient and weakly efficient point sets. These statements include relationships between the efficient point set of $F$ and the efficient point set of $F+D$. In applications, $F+D$ may be closed or convex though $F$ does not have this property. 

The definitions imply (\cite{Wei83},\cite{Wei85}):

\begin{lemma}\label{l-Eff-prop}
Assume that $Y$ is a linear space, $F, D\subseteq Y$.
Then we have:
\begin{itemize}
\item[(a)] $\operatorname*{Eff}(F,D)=\operatorname*{Eff}(F,D\cup\{0\})=\operatorname*{Eff}(F,D\setminus\{0\})$.
\item[(b)] $D_1\subseteq D\implies \operatorname*{Eff}(F,D)\subseteq \operatorname*{Eff}(F,D_1)$.
\item[(c)] $F_1\subseteq F\implies \operatorname*{Eff}(F,D)\cap F_1\subseteq \operatorname*{Eff}(F_1,D)$.
\item[(d)] Suppose $F\subseteq A\subseteq F+(D\cup\{0\})$.
Then $\operatorname*{Eff}(A,D)\subseteq \operatorname*{Eff}(F,D)$.\\
If, additionally, $D\cap (-D)\subseteq \{0\}$ and $D+D\subseteq D$, then\\
$\operatorname*{Eff}(A,D)=\operatorname*{Eff}(F,D)$.
\item[(e)] If $D+D\subseteq D$, then $\operatorname*{Eff}(F\cap (y-D),D)=\operatorname*{Eff}(F,D)\cap (y-D)\mbox{ for all } y\in Y$.
\end{itemize}
\end{lemma}

\begin{proof}
\begin{itemize}
\item[]
\item[(a)]-(c) follow immediately from the definition of efficient elements.
\item[(d)] Consider some $y^0\in \operatorname*{Eff}(A,D)$.
Since $A\subseteq F+(D\cup\{0\})$, there exist $y^1\in F$, $d\in D\cup\{0\}$ with $y^0=y^1+d$.
$\Rightarrow y^1\in y^0-(D\cup\{0\})$. $\Rightarrow y^1=y^0$ because of $y^0\in \operatorname*{Eff}(A,D)$ and $F\subseteq A$.
$\Rightarrow y^0\in F$. Hence, $\operatorname*{Eff}(A,D)\subseteq F$.
$\operatorname*{Eff}(A,D)\subseteq \operatorname*{Eff}(F,D)$ follows from (c) since $F\subseteq A$.\\
Assume now $D\cap (-D)\subseteq \{0\}$ and $D+D\subseteq D$. Suppose that $\operatorname*{Eff}(A,D)\not=\operatorname*{Eff}(F,D)$. $\Rightarrow \exists y^0\in\operatorname*{Eff}(F,D)\setminus\operatorname*{Eff}(A,D)$.
$\Rightarrow \exists a\in A:\;a\in y^0-(D\setminus\{0\})$.
$ A\subseteq F+(D\cup\{0\}) \Rightarrow \exists y\in F, d\in D\cup\{0\} :\;y+d\in y^0-(D\setminus\{0\})$.
$\Rightarrow y^0-y\in d+(D\setminus\{0\})\subseteq D$. $\Rightarrow y^0=y$ since $y^0\in\operatorname*{Eff}(F,D)$.
Then $y+d\in y^0-(D\setminus\{0\})$ implies $d\in -(D\setminus\{0\})$, a contradiction to $D\cap (-D)\subseteq \{0\}$. 
\item[(e)] Choose some $y\in Y$.
$\operatorname*{Eff}(F\cap (y-D),D)\supseteq\operatorname*{Eff}(F,D)\cap (y-D)$ results from (c).
Assume there exists some $y^0\in\operatorname*{Eff}(F\cap (y-D),D)\setminus\operatorname*{Eff}(F,D)$.
$\Rightarrow (y^0-D)\cap (F\cap (y-D))\subseteq \{y^0\}$ and $\exists y^1\not=y^0:\; y^1\in F\cap(y^0-D)$.
$y^1\in y^0-D\subseteq (y-D)-D\subseteq y-D$. $\Rightarrow y^1\in (y^0-D)\cap (F\cap (y-D))$, a contradiction.
\end{itemize}
\end{proof}

\begin{remark}
$\operatorname*{Eff}(F+D,D)\subseteq \operatorname*{Eff}(F,D)$ was proved in $Y=\mathbb{R}^\ell$ by Bergstresser et al. \cite[Lemma 2.2]{Berg76} for convex sets $D$ with $0\in D\setminus \operatorname*{int}D$, by Sawaragi, Nakayama and Tanino 
\cite[Prop. 3.1.2]{sana85} for cones $D$. Vogel \cite[Satz 6]{Vog75} showed
$\operatorname*{Eff}(F+D,D)= \operatorname*{Eff}(F,D)$ for pointed convex cones $D$.
Yu \cite[p.20]{yu74} gave an example for a convex cone which is not pointed and for which $\operatorname*{Eff}(F,D)$ is not a subset of $\operatorname*{Eff}(F+D,D)$.
\end{remark}

We get for weakly efficient elements \cite{Wei85}:

\begin{lemma}\label{l-wEff-prop}
Assume that $Y$ is a topological vector space, $F, D\subseteq Y$.
Then:
\begin{itemize}
\item[(a)] $\operatorname*{Eff}(F,D)\subseteq \operatorname*{WEff}(F,D)$.
\item[(b)] $\operatorname*{int}D_1\subseteq \operatorname*{int}D\implies \operatorname*{WEff}(F,D)\subseteq \operatorname*{WEff}(F,D_1)$.
\item[(c)] $F_1\subseteq F\implies \operatorname*{WEff}(F,D)\cap F_1\subseteq \operatorname*{WEff}(F_1,D)$.
\item[(d)] Suppose $F\subseteq A\subseteq F\cup (F+\operatorname*{int}D)$.
Then $\operatorname*{WEff}(A,D)\subseteq \operatorname*{WEff}(F,D)$.\\
If, additionally, $\operatorname*{int}D\cap (-\operatorname*{int}D)\subseteq \{0\}$ and $\operatorname*{int}D+\operatorname*{int}D\subseteq D$, then\\
$\operatorname*{WEff}(A,D)=\operatorname*{WEff}(F,D)$.
\item[(e)] Suppose $F\subseteq A\subseteq F+(D\cup\{0\})$ and $D+\operatorname*{int}D\subseteq D\setminus\{0\}$.\\
Then $\operatorname*{WEff}(A,D)\cap F=\operatorname*{WEff}(F,D)$.
\item[(f)] If $D+\operatorname*{int}D\subseteq D$, then\\
$\operatorname*{WEff}(F\cap (y-D),D)=\operatorname*{WEff}(F,D)\cap (y-D)\mbox{ for all } y\in Y$.
\item[(g)] $\operatorname*{WEff}(F,D)=F$ if $\operatorname*{int}D=\emptyset $.
\item[(h)] $\operatorname*{WEff}(F,D)=\operatorname*{WEff}(F,\operatorname*{cl}D)$ if $D$ is convex and $\operatorname*{int}D\not=\emptyset $.
\end{itemize}
\end{lemma}

\begin{proof}
\begin{itemize}
\item[]
\item[(a)] follows from Lemma \ref{l-Eff-prop}(b) with $D_1=\operatorname*{int}D$.
\item[(b)]- (d) result from Lemma \ref{l-Eff-prop}(b)-(d) when replacing $D$ and $D_1$ by their interior.
\item[(e)] Consider some $y^0\in F$ with $y^0\notin \operatorname*{WEff}(A,D)$. $\Rightarrow \exists a\in A\setminus\{y^0\}:\;\; y^0\in a+\operatorname*{int}D$.
Since $A\subseteq F+(D\cup\{0\})$, there exist $y^1\in F, d\in D\cup\{0\}$ such that $y^0\in y^1+d+\operatorname*{int}D$.
$D+\operatorname*{int}D\subseteq D\setminus\{0\}$ implies $0\notin \operatorname*{int}D$ and $y^0\in y^1+(\operatorname*{int}D\setminus\{0\})$. $\Rightarrow y^0\notin \operatorname*{WEff}(F,D)$.\\
Hence, $\operatorname*{WEff}(F,D)\subseteq \operatorname*{WEff}(A,D)$. The assertion follows by (c) since $F\subseteq A$.
\item[(f)] Choose some $y\in Y$.
$\operatorname*{WEff}(F\cap (y-D),D)\supseteq\operatorname*{WEff}(F,D)\cap (y-D)$ results from (c).
Assume there exists some $y^0\in\operatorname*{WEff}(F\cap (y-D),D)\setminus\operatorname*{WEff}(F,D)$.
$\Rightarrow (y^0-\operatorname*{int}D)\cap (F\cap (y-D))\subseteq \{y^0\}$ and $\exists y^1\not=y^0:\; y^1\in F\cap(y^0-\operatorname*{int}D)$.
$y^1\in y^0-\operatorname*{int}D\subseteq (y-D)-\operatorname*{int}D\subseteq y-D$. $\Rightarrow y^1\in (y^0-\operatorname*{int}D)\cap (F\cap (y-D))$, a contradiction.
\item[(g)] is obvious.
\item[(h)] Under the assumptions, $\operatorname*{int}D=\operatorname*{int}\operatorname*{cl}D$.
\end{itemize}
\end{proof}

\begin{corollary}\label{c-FplusD}
Assume that $Y$ is a linear space, $F\subseteq Y$, that $D\subseteq Y$ is a non-trivial pointed convex cone and $F\subseteq A\subseteq F+D$.
\begin{itemize}
\item[(a)] $\operatorname*{Eff}(A,D)=\operatorname*{Eff}(F,D)$.
\item[(b)]If $Y$ is a linear topological space, then $\operatorname*{WEff}(A,D)\cap F=\operatorname*{WEff}(F,D)$.
\end{itemize}
\end{corollary}

Podinovskij and Nogin \cite[Lemma 2.2.1]{pono82} proved part (b) of Corollary \ref{c-FplusD} in $Y=\mathbb{R}^\ell$ for $D=\mathbb{R}^\ell _+$ and $A=F+\mathbb{R}^\ell _+$. They gave the following example that, in general, $\operatorname*{WEff}(F+\mathbb{R}^\ell _+,D)=\operatorname*{WEff}(F,\mathbb{R}^\ell _+)$ is not fulfilled.

\begin{example}
$Y:=\mathbb{R}^2$, $F:=\{(y_1,y_2)\in Y\mid y_2>0\}\cup\{0\}$.
Then $F+\mathbb{R}^2_+=F\cup\{(y_1,y_2)\in Y\mid y_1>0, y_2=0\}$.
$(1,0)^T\in \operatorname*{WEff}(F+\mathbb{R}^2_+,\mathbb{R}^2_+)\setminus \operatorname*{WEff}(F,\mathbb{R}^2_+)$ since $(1,0)^T\notin F$.
\end{example}

Efficient elements are often located on the boundary of the feasible point set.

\begin{theorem}\label{eff-in-bd}
Assume $Y$ is a topological vector space, $F, D\subseteq Y$ and\\
$0\in\operatorname*{bd}(D\setminus\{0\})$. Then
\begin{itemize}
\item[(a)] $\operatorname*{Eff}(F,D)\subseteq\operatorname*{bd}F.$
\item[(b)] If, moreover, $D\cap (-D)\subseteq \{0\}$ and $D+D\subseteq D$, then
$$\operatorname*{Eff}(F,D)\subseteq\operatorname*{bd}A$$
for each set $A\subseteq Y$ with $F\subseteq A\subseteq F+(D\cup\{0\}).$
\end{itemize}
\end{theorem}
\begin{proof}
\begin{itemize}
\item[]
\item[(a)] Suppose there exists some $y^0\in \operatorname*{Eff}(F,D)\setminus\operatorname*{bd}F$.
$\Rightarrow y^0\in \operatorname*{int}F$. Hence, there exists some neighborhood $U$ of $y^0$ with $U\subseteq F$.
$\Rightarrow W:=y^0-U$ is a neighborhood of $0$.
Since $0\in\operatorname*{bd}(D\setminus\{0\})$, there exists some $d\in W\cap(D\setminus\{0\})\subseteq (y^0-F)\cap(D\setminus\{0\})$.
$\Rightarrow \exists y\in F:\; d=y^0-y\in D\setminus\{0\}$.
$\Rightarrow y\in y^0-(D\setminus\{0\})$, a contradiction.
\item[(b)] results from Lemma \ref{l-Eff-prop}. 
\end{itemize}
\end{proof}

\begin{remark}
Theorem \ref{eff-in-bd} was proved in \cite{Wei85}. There also references to earlier results in $Y=\mathbb{R}^\ell$ for more special sets $D$ are given as well as illustrating examples.
\end{remark}

\begin{proposition}\label{eff-in-algcl}
Assume $Y$ is a linear space, $F, D\subseteq Y$, and that there exists some $d\in D\setminus\{0\}$ such that $td\in D\setminus\{0\}\mbox{ for all } t\in (0,1)$. Then
$$\operatorname*{Eff}(F,D)\subseteq F\setminus\operatorname*{core}F.$$
\end{proposition}
\begin{proof}
Consider some $y^0\in \operatorname*{core}F$.
$\Rightarrow \exists t\in\mathbb{R}_>:\; t<1$ and $y^0-td\in F$. $\Rightarrow y^0\notin \operatorname*{Eff}(F,D)$. 
\end{proof}

The assumption is fulfilled if $D\setminus\{0\}\not=\emptyset$ and $D\cup\{0\}$ is star-shaped about zero. This is the case if $D$ is a non-trivial cone or $D\cup\{0\}$ is convex and $D\setminus\{0\}\not=\emptyset$.

We get for weakly efficient points \cite{Wei85}:

\begin{theorem}\label{weff-in-bd}
Let $Y$ be a topological vector space, $F, D\subseteq Y$. Then
\begin{equation}\label{bd-in}
F\cap\operatorname*{bd}(F+D)\subseteq\operatorname*{WEff}(F,D).
\end{equation}
Assume now $0\in\operatorname*{bd}(\operatorname*{int}D\setminus\{0\})$.
\begin{itemize}
\item[(a)] $\operatorname*{WEff}(F,D)\subseteq\operatorname*{bd}F.$
\item[(b)] If $D+\operatorname*{int}D\subseteq D\setminus\{0\}$,then
$$\operatorname*{WEff}(F,D)=F\cap\operatorname*{bd}(F+D)$$
and
$$\operatorname*{WEff}(F,D)\subseteq\operatorname*{bd}A$$
for each set $A\subseteq Y$ with $F\subseteq A\subseteq F+(D\cup\{0\}).$
\item[(c)] If $\operatorname*{int}D\cap (-\operatorname*{int}D)\subseteq \{0\}$ and $\operatorname*{int}D+\operatorname*{int}D\subseteq D$, then
$$\operatorname*{WEff}(F,D)\subseteq\operatorname*{bd}A$$
for each set $A\subseteq Y$ with $F\subseteq A\subseteq F\cup(F+\operatorname*{int}D)$.
\end{itemize}
\end{theorem}
\begin{proof}
If $y^0\in F\setminus\operatorname*{WEff}(F,D)$, then there exists some $y\in F\cap(y^0-\operatorname*{int}D)$, which implies $y^0\in\operatorname*{int}(F+D)$. Thus, (\ref{bd-in}) holds.
\begin{itemize}
\item[(a)] and (c) follow immediately from Theorem \ref{eff-in-bd}.
\item[(b)] The second statement results from (a) and Lemma \ref{l-wEff-prop}(e) since\\ $\operatorname*{WEff}(F,D)\subseteq\operatorname*{WEff}(A,D)\subseteq\operatorname*{bd}A$, the first one from the second by (\ref{bd-in}).
\end{itemize}
\end{proof}

\begin{theorem}\label{weff-nocore}
Assume $Y$ is a topological vector space, $F, D\subseteq Y$.
\begin{itemize}
\item[(a)] We have 
\begin{equation}\label{core-in}
F\setminus\operatorname*{core}(F+D)\subseteq\operatorname*{WEff}(F,D).
\end{equation}
\item[(b)] If $D$ is a non-trivial convex cone with nonempty interior, then
$$\operatorname*{WEff}(F,D)=F\setminus\operatorname*{core}(F+D).$$
\end{itemize}
\end{theorem}
\begin{proof}
\begin{itemize}
\item[]
\item[(a)]If $y^0\in F\setminus\operatorname*{WEff}(F,D)$, then there exists some $y\in F\cap(y^0-\operatorname*{core}D)$, which implies $y^0\in\operatorname*{core}(F+D)$. Thus, (\ref{core-in}) holds.
\item[(b)] $\operatorname*{WEff}(F,D)=\operatorname*{WEff}(F+D,D)\cap F$ by Lemma \ref{l-wEff-prop}. $\operatorname*{WEff}(F+D,D)\subseteq (F+D)\setminus\operatorname*{core}(F+D)$ by Proposition \ref{eff-in-algcl}.
\end{itemize}
\end{proof}

\begin{corollary}
Assume $Y$ is a topological vector space, $F\subseteq Y$, and that $D\subset Y$ is a non-trivial convex pointed cone. Then
\begin{eqnarray*}
\operatorname*{Eff}(F,D) & \subseteq & \operatorname*{bd}F\quad\mbox{ and }\\
\operatorname*{Eff}(F,D) & \subseteq & \operatorname*{bd}(F+D).
\end{eqnarray*}
If $D$ has a nonempty interior, then
\begin{eqnarray*}
\operatorname*{WEff}(F,D) & \subseteq & \operatorname*{bd}F\quad\mbox{ and }\\
\operatorname*{WEff}(F,D) & = & F\cap\operatorname*{bd}(F+D).
\end{eqnarray*}
\end{corollary}

Theorem \ref{weff-in-bd} delivers the following statement about the existence of weakly efficient elements.

\begin{corollary}
Assume $Y$ is a topological vector space, $F, D\subseteq Y$. Then
$$\operatorname*{WEff}(F,D)\not=\emptyset$$
if $F\cap\operatorname*{bd}(F+D)\not=\emptyset$.\\
This condition is fulfilled if $F+D$ is a proper closed subset of $Y$, $D=D+D$, and $0\in D$.
\end{corollary}
\begin{proof}
The first statement follows immediately from (\ref{bd-in}).\\
$F+D\not= Y$ is closed $\Rightarrow \exists\;y\in F, d\in D:\;y+d\in\operatorname*{bd}(F+D)$.
If $y\in\operatorname*{int}(F+D)$, then $y+d\in d+\operatorname*{int}(F+D)\subseteq\operatorname*{int}(F+D)$ because of $D+D=D$, a contradiction.
This and $y\in F\subseteq F+D$ imply $y\in F\cap\operatorname*{bd}(F+D)$. 
\end{proof}

An existence result for efficient elements will be proved using scalarizing functionals.

We are now going to point out which properties a functional must have in order to be appropriate for the scalarization of vector optimization problems. We will consider a functional as being appropriate for scalarization if its optimal solutions are related to solutions of the vector optimization problem.

For $\varphi :M\to \overline{\mathbb{R}}_{\nu }$, $M$ being an arbitrary set,  we denote the \textbf{set of minimal solutions} of $\varphi$ on $M$ as
\begin{displaymath}
\operatorname*{argmin}_{y\in M}\varphi (y):= \left \{
\begin{array}{cl}
\{ \bar{y}\in M \mid \varphi (\bar{y})=\min\limits_{y\in M}\, \varphi (y)\} & \mbox{ if } \varphi \mbox{ attains a minimum on } M, \\
\emptyset  & \mbox{otherwise}.
\end{array} \right.
\end{displaymath} 
We will also use the abbreviation $\operatorname*{argmin}\nolimits_{M}\varphi :=\operatorname*{argmin}\nolimits_{y\in M}\varphi (y)$.

We can show \cite{Wei90}:

\begin{proposition}\label{p-mon-scal}
Assume that $Y$ is a linear space, $F, D\subseteq Y$ and $\varphi :F\to \mathbb{R}$.
\begin{itemize}
\item[(a)] $\operatorname*{Eff}(F,D)\cap \operatorname*{argmin}\nolimits_{F}\varphi \subseteq \operatorname*{Eff}(\operatorname*{argmin}\nolimits_{F}\varphi ,D)$.
\item[(b)] If $\varphi $ is $D$--monotone on $F$, then $\operatorname*{Eff}(F,D)\cap \operatorname*{argmin}\nolimits_{F}\varphi = \operatorname*{Eff}(\operatorname*{argmin}\nolimits_{F}\varphi ,D)$.\\
If, additionally, $\operatorname*{argmin}\nolimits_{F}\varphi =\{y^0\}$, then $y^0\in \operatorname*{Eff}(F,D)$.
\item[(c)] $\operatorname*{argmin}\nolimits_{F}\varphi \subseteq \operatorname*{Eff}(F,D)$ holds if $\varphi$ is strictly $D$--monotone on $F$.
\end{itemize}
\end{proposition}

\begin{proof}
\begin{itemize}
\item[]
\item[(a)] results from Lemma \ref{l-Eff-prop}(c) since $\operatorname*{argmin}\nolimits_{F}\varphi \subseteq F$.
\item[(b)] Consider some $y^0\in \operatorname*{Eff}(\operatorname*{argmin}\nolimits_{F}\varphi ,D)$ and assume\\
$y^0\notin \operatorname*{Eff}(F,D)\cap \operatorname*{argmin}\nolimits_{F}\varphi $.
$\Rightarrow y^0\notin \operatorname*{Eff}(F,D)$. $\Rightarrow \exists y\not= y^0:\; y\in F\cap (y^0-D)$.
$\Rightarrow y^0-y\in D$. $\Rightarrow \varphi (y^0)\geq \varphi (y)$ since $\varphi $ is $D$--monotone on $F$.
$\Rightarrow y\in \operatorname*{argmin}\nolimits_{F}\varphi \cap (y^0-D)$. $\Rightarrow y^0\notin \operatorname*{Eff}(\operatorname*{argmin}\nolimits_{F}\varphi ,D)$, a contradiction.\\
If $\operatorname*{argmin}\nolimits_{F}\varphi =\{y^0\}$, then $\operatorname*{Eff}(\operatorname*{argmin}\nolimits_{F}\varphi ,D)=\{y^0\}$, which yields the assertion.
\item[(c)] Consider some $y^0\in \operatorname*{argmin}\nolimits_{F}\varphi $ and assume $y^0\notin \operatorname*{Eff}(F,D)$.
$\Rightarrow \exists y\not= y^0:\; y\in F\cap (y^0-D)$.
$\Rightarrow y^0-y\in D\setminus\{0\}$. $\Rightarrow \varphi (y^0)> \varphi (y)$ since $\varphi $ is strictly $D$--monotone on $F$, a contradiction to $y^0\in \operatorname*{argmin}\nolimits_{F}\varphi $.

\end{itemize}
\end{proof}

Immediately from the previous proposition, we get the related statements for weakly efficient elements.

\begin{proposition}\label{p-mon-wscal}
Assume that $Y$ is a topological vector space, $F, D\subseteq Y$ and $\varphi :F\to \mathbb{R}$.
\begin{itemize}
\item[(a)] $\operatorname*{WEff}(F,D)\cap \operatorname*{argmin}\nolimits_{F}\varphi \subseteq \operatorname*{WEff}(\operatorname*{argmin}\nolimits_{F}\varphi ,D)$.
\item[(b)] If $\varphi $ is $(\operatorname*{int}D)$-monotone on $F$, then\\
$\operatorname*{WEff}(F,D)\cap \operatorname*{argmin}\nolimits_{F}\varphi = \operatorname*{WEff}(\operatorname*{argmin}\nolimits_{F}\varphi ,D)$.\\
If, additionally, $\operatorname*{argmin}\nolimits_{F}\varphi =\{y^0\}$, then $y^0\in \operatorname*{WEff}(F,D)$.
\item[(c)] $\operatorname*{argmin}\nolimits_{F}\varphi \subseteq \operatorname*{WEff}(F,D)$ holds if $\varphi$ is strictly $(\operatorname*{int}D)$-monotone on $F$.
\end{itemize}
\end{proposition}

For efficient elements, we can now prove the following existence result.

\begin{theorem}\label{ext-min}
Assume that $Y$ is a topological vector space, 
\begin{itemize}
\item[(i)] $D\subset Y$ is a set for which there exists some proper closed subset $H$ of $Y$ with
$\operatorname*{core}0^+H\not=\emptyset$ and $H-(D\setminus\{0\})\subseteq 
\operatorname*{core}H$,
\item[(ii)] $F\subseteq Y$ is a nonempty compact set or ($D+D\subseteq D$ and there exists some $y\in Y$ such that $F\cap (y-D)$ is a nonempty compact set).
\end{itemize}
Then
$$\operatorname*{Eff}(F,D)\not=\emptyset.$$
Each set $D$ for which $D\setminus\{0\}$ is contained in the core of a non-trivial closed convex cone $C$ fulfills the assumption (i).
\end{theorem}
\begin{proof}
\begin{itemize}
\item[]
\item[(a)] First we assume that $F$ is nonempty and compact.\\
Choose some $k\in -\operatorname*{core}0^+H$. $(H1_{H,k})$ is fulfilled, $\varphi _{H,k}$ is lower semicontinuous and finite-valued because of Lemma \ref{t251ua}. Hence, $\varphi _{H,k}$ attains a minimum on $F$. $\varphi _{H,k}$ is strictly $D$-monotone by Lemma \ref{t251ua}. Hence, the minimizers of $\varphi _{H,k}$ belong to $\operatorname*{Eff}(F,D)$ because of Proposition \ref{p-mon-scal}.
\item[(b)] Assume now $D+D\subseteq D$ and that $F\cap (y-D)$ is nonempty and compact. $\operatorname*{Eff}(F\cap (y-D),D)\not=\emptyset$ because of (a).
By Lemma \ref{l-Eff-prop}$, \operatorname*{Eff}(F\cap (y-D),D)=\operatorname*{Eff}(F,D)\cap (y-D)$, which yields $\operatorname*{Eff}(F,D)\not=\emptyset$.
\end{itemize}
For the final statement, take $H=-C$.
\end{proof}

Further statements about the existence of efficient elements are given in \cite{Wei85}. A comprehensive study of existence results for optimal elements w.r.t. relations and to sets is contained in \cite{GopRiaTamZal:03}.
\smallskip

\section{Scalarization of the Efficient and the Weakly Efficient Point Set by Functionals with Uniform Sublevel Sets}\label{s-scal-uni}

We will now derive conditions for efficient and weakly efficient elements by the functionals with uniform sublevel sets which we have investigated in \cite{Wei16a} on the base of the lemmata given in the previous section. Here, we use functionals $\varphi _{a-H,k}$, where $a\in Y$ can be considered to be some reference point and $H\subset Y$ is a set related to the domination set $D$.

We will use the following supposition:

\quad

\begin{tabular}{ll}
(H1--VOP$_{H,k}$): & $Y$ is a topological vector space, $F, D\subset Y$, and\\
& $H$ is a closed proper subset of $Y$ with $k\in 0^+H\setminus \{0\}$. 
\end{tabular}

\quad

Note that (H1--VOP$_{H,k}$) implies $(H1_{-H,k})$ and $(H1_{a-H,k})$ for each $a\in Y$ and that $\varphi _{a-H,k}$ is finite-valued if $k\in\operatorname*{core}0^+H$.  

Even if the functionals $\varphi _{a-H,k}$ are not defined on the whole set $F$, they can deliver efficient and weakly efficient elements of $F$.

\begin{lemma}\label{hab-l421}
Suppose (H1--VOP$_{H,k}$), $a\in Y$. 
\begin{itemize}
\item[(a)] $\operatorname*{Eff}(F,D)\cap \operatorname*{dom}\varphi _{a-H,k}\subseteq \operatorname*{Eff}(F\cap\operatorname*{dom}\varphi _{a-H,k},D)$.
\item[(b)] $H+D\subseteq H\implies  \operatorname*{Eff}(F,D)\cap \operatorname*{dom}\varphi _{a-H,k}= \operatorname*{Eff}(F\cap\operatorname*{dom}\varphi _{a-H,k},D)$.
\item[(c)] $\operatorname*{WEff}(F,D)\cap \operatorname*{dom}\varphi _{a-H,k}\subseteq \operatorname*{WEff}(F\cap\operatorname*{dom}\varphi _{a-H,k},D)$.
\item[(d)] $H+\operatorname*{int}D\subseteq H\implies  \operatorname*{WEff}(F,D)\cap \operatorname*{dom}\varphi _{a-H,k}= \operatorname*{WEff}(F\cap\operatorname*{dom}\varphi _{a-H,k},D)$.
\end{itemize}
\end{lemma}

\begin{proof}
\begin{itemize}
\item[]
\item[(a)] results from Lemma \ref{l-Eff-prop}(c).
\item[(b)] Consider an arbitrary $y^0\in F\cap\operatorname*{dom}\varphi _{a-H,k}$. $\Rightarrow \exists t\in\mathbb{R}:\;y^0\in a-H+tk$. Assume $y^0\notin\operatorname*{Eff}(F,D)$. $\Rightarrow \exists y\in F\cap (y^0-(D\setminus\{0\})$. $\Rightarrow y\in a-H+tk-D\subseteq a+tk-H\subseteq \operatorname*{dom}\varphi _{a-H,k}$.
$\Rightarrow y^0\notin\operatorname*{Eff}(F\cap\operatorname*{dom}\varphi _{a-H,k},D)$. 
\item[(c)] and (d) follow from (a) and (b) with $\operatorname*{int}D$ instead of $D$.
\end{itemize}
\end{proof}

Let us now first give some sufficient conditions for efficient and weakly efficient points by minimal solutions of functions $\varphi _{a-H,k}$. Keep in mind that $F\cap\operatorname*{dom}\varphi _{a-H,k}$ is the feasible range of $\operatorname*{min}_{y\in F}\varphi _{a-H,k}(y)$.

\begin{theorem}\label{hab-t421}
Suppose (H1--VOP$_{H,k}$), $a\in Y$. Define 
$$\Psi:=\operatorname*{argmin}_{y\in F}\varphi _{a-H,k}(y).$$
Then:
\begin{itemize}
\item[(a)] $\operatorname*{Eff}(F,D)\cap \Psi\subseteq \operatorname*{Eff}(\Psi,D)$.
\item[(b)] $H+D\subseteq H\implies  \operatorname*{Eff}(F,D)\cap \Psi= \operatorname*{Eff}(\Psi,D)$.
\item[(c)] $H+D\subseteq H$ and $\Psi=\{ y^0\}$ imply $y^0\in\operatorname*{Eff}(F,D)$.
\item[(d)] Assume that $F\subseteq \operatorname*{dom}\varphi _{a-H,k}$ is convex and that
$H$ is a strictly convex set with $H+D\subseteq H$ and $H+\mathbb{R}_>k\subseteq \operatorname*{int}H$. Then $\Psi\subseteq\operatorname*{Eff}(F,D)$.\\
In this case, $\Psi$ contains at most one element, which has to be an extreme point of
$F$ and is the only local minimizer of $\varphi _{a-H,k}$ on
$F$.
\item[(e)] $H+(D\setminus\{0\})\subseteq \operatorname*{core}H\implies \Psi\subseteq \operatorname*{Eff}(F,D)$.
\item[(f)] $H+\operatorname*{int}D\subseteq H\implies  \Psi\subseteq \operatorname*{WEff}(F,D)$.
\end{itemize}
\end{theorem}

\begin{proof}
\begin{itemize}
\item[]
\item[(a)] follows from Lemma \ref{l-Eff-prop}(c).
\item[(b)]$H+D\subseteq H\Rightarrow \varphi _{-H,k}$ is $D$-monotone by Lemma \ref{t251ua}. $\Rightarrow \varphi _{a-H,k}$ is $D$-monotone because of Lemma \ref{A-shift}. $\Rightarrow \operatorname*{Eff}(F\cap\operatorname*{dom}\varphi _{a-H,k},D)\cap \Psi= \operatorname*{Eff}(\Psi,D)$ by Proposition
\ref{p-mon-scal}(b). This results in the assertion by Lemma \ref{hab-l421}(b).
\item[(c)] follows immediately from (b).
\item[(d)] $\varphi _{a-H,k}$ is strictly quasiconvex by Lemma \ref{prop-funcII}.
This implies the assertion. 
\item[(e)] If $\varphi _{a-H,k}$ is not finite-valued on $F\cap\operatorname*{dom}\varphi _{a-H,k}$, the assertion is fulfilled. Assume now that $\varphi _{a-H,k}$ is finite-valued on $F\cap\operatorname*{dom}\varphi _{a-H,k}$ and $H+(D\setminus\{0\})\subseteq \operatorname*{core}H$. Then $\varphi _{a-H,k}$ is strictly $D$-monotone on $F\cap\operatorname*{dom}\varphi _{a-H,k}$ by Lemma \ref{t251ua}. $\Rightarrow \Psi\subseteq \operatorname*{Eff}(F\cap\operatorname*{dom}\varphi _{a-H,k},D)$ by Proposition \ref{p-mon-scal}(c). The assertion follows by Lemma \ref{hab-l421}(b).
\item[(f)] Since $H+\operatorname*{int}D\subseteq \operatorname*{int}H\subseteq \operatorname*{core}H$, (f) results from (e) with $D$ replaced by $\operatorname*{int}D$.
\end{itemize}
\end{proof}

\begin{example}
$Y=\mathbb{R}^2$, $H=D=\mathbb{R}_+^2+(1,1)^T$ and $k=(1,1)^T$ fulfill the assumptions for $H$ and $k$ in (H1--VOP$_{H,k}$) and $H+D\subseteq H$, though $D$ is not a convex cone and $0\notin \operatorname*{bd}D$.
\end{example}

\begin{corollary}\label{c-hab-t421}
Suppose that $Y$ is a topological vector space, $F\subset Y$, $a\in Y$, that $D$ is a non-trivial closed convex cone in $Y$, $k\in \operatorname*{int}D$. Define 
$$\Psi:=\operatorname*{argmin}_{y\in F}\varphi _{a-D,k}(y).$$
Then $(H2_{-D,k})$ and $(H2_{a-D,k})$ hold and:
\begin{itemize}
\item[(a)] $\Psi\subseteq \operatorname*{WEff}(F,D)$.
\item[(b)] $\operatorname*{Eff}(F,D)\cap \Psi= \operatorname*{Eff}(\Psi,D)$.
\item[(c)] $\Psi=\{ y^0\} \implies y^0\in\operatorname*{Eff}(F,D)$.
\end{itemize}
\end{corollary}

We will now characterize the efficient point set and the weakly efficient point set by minimal solutions of functionals
$\varphi _{a-D,k}$. The following two theorems deliver necessary conditions for weakly efficient and for efficient elements.

\begin{theorem}\label{hab-p431a}
Suppose (H1--VOP$_{D,k}$).\\ 
Then $(H1_{-D,k})$ and for each $y^0\in Y$ also $(H1_{y^0-D,k})$ are fulfilled, and
\begin{eqnarray*}
\operatorname*{Eff}(F,D) & = & \{y^0\in F\mid \forall y\in (F\cap\operatorname*{dom}\varphi _{y^0-D,k})\setminus\{y^0\}\colon\varphi _{y^0-D,k}(y)>0\}\\
 & = & \{y^0\in F\mid \forall y\in (F\cap\operatorname*{dom}\varphi _{y^0-D,k})\setminus\{y^0\}\colon\varphi _{-D,k}(y-y^0)>0\}.
\end{eqnarray*} 
If $0\in D$, then $\varphi _{y^0-D,k}(y^0)=\varphi _{-D,k}(y^0-y^0)\leq 0$ for each $y^0\in Y$.\\
If $D+\mathbb{R}_>k\subseteq\operatorname*{int}D$ and $0\in \operatorname*{bd}D$, then\\
$\varphi _{y^0-D,k}(y^0)=\varphi _{-D,k}(y^0-y^0)=0$ for each $y^0\in Y$ and\\
$\operatorname*{Eff}(F,D) = \{y^0\in F\mid  \forall y\in (F\cap\operatorname*{dom}\varphi _{y^0-D,k})\setminus\{y^0\}\colon\varphi _{y^0-D,k}(y^0) < \varphi _{y^0-D,k}(y)\},$\\
$\operatorname*{Eff}(F,D) = \{y^0\in F\mid  \forall y\in (F\cap\operatorname*{dom}\varphi _{y^0-D,k})\setminus\{y^0\}\colon\varphi _{-D,k}(y^0-y^0) < \varphi _{-D,k}(y-y^0)\}.$
\end{theorem}

\begin{proof}
For the relationship between $\varphi _{-D,k}$ and $\varphi _{y^0-D,k}$, see Lemma \ref{A-shift}.\\
Consider some arbitrary $y^0\in F$. $y^0-D\subseteq \operatorname*{dom}\varphi _{y^0-D,k}$.\\
$y^0\in\operatorname*{Eff}(F,D)\Leftrightarrow F\cap (y^0-D)\subseteq\{y^0\}\Leftrightarrow\varphi _{y^0-D,k}(y)>0\mbox{ for all } y\in (F\cap\operatorname*{dom}\varphi _{y^0-D,k})\setminus\{y^0\}$.\\ 
$0\in D\Rightarrow y^0\in y^0-D.\Rightarrow \varphi _{y^0-D,k}(y^0)\leq 0$.\\
$D+\mathbb{R}_>k\subset\operatorname*{int}D$ implies $(H2_{y^0-D,k})$ and $(H2_{-D,k})$, and this yields the assertion.
\end{proof}

$\varphi _{-D,k}$ and $\varphi _{y^0-D,k}$ in Theorem \ref{hab-p431a} are lower semicontinuous and even continuous in the case $D+\mathbb{R}_>k\subset\operatorname*{int}D$. Moreover, because of Lemma \ref{t251ua}, the functionals $\varphi _{y^0-D,k}$ in Theorem \ref{hab-p431a} and in the following theorem are convex if and only if $D$ is a convex set, and the functionals $\varphi _{-D,k}$ in these theorems are sublinear if and only if $D$ is a convex cone. Note that $\varphi _{-D,k}$ and each $\varphi _{y^0-D,k}$ are finite-valued if $k\in\operatorname*{core}0^+D$.

In Theorem \ref{hab-p431a}, efficient elements $y^0$ are described as unique minimizers of $\varphi _{y^0-D,k}$. Without the uniqueness, we get weakly efficient points.

\begin{theorem}\label{hab-p431b}
Suppose (H1--VOP$_{D,k}$) and $D+\mathbb{R}_>k\subseteq\operatorname*{int}D$.\\ 
Then $(H2_{-D,k})$ and for each $y^0\in Y$ also $(H2_{y^0-D,k})$ are fulfilled, and
\begin{eqnarray*}
\operatorname*{WEff}(F,D) & = & \{y^0\in F\mid \forall y\in (F\cap\operatorname*{dom}\varphi _{y^0-D,k})\setminus\{y^0\}\colon\varphi _{y^0-D,k}(y)\geq 0\}\\
 & = & \{y^0\in F\mid \forall y\in (F\cap\operatorname*{dom}\varphi _{y^0-D,k})\setminus\{y^0\}\colon\varphi _{-D,k}(y-y^0)\geq 0\}.
\end{eqnarray*} 
If $0\in D$, then $\varphi _{y^0-D,k}(y^0)=\varphi _{-D,k}(y^0-y^0)\leq 0$ for each $y^0\in Y$.\\
If $0\in\operatorname*{bd}D$, then $\varphi _{y^0-D,k}(y^0)=\varphi _{-D,k}(y^0-y^0)=0$ for each $y^0\in Y$ and
\begin{eqnarray*}
\operatorname*{WEff}(F,D) & = & \{y^0\in F\mid \varphi _{y^0-D,k}(y^0)=\operatorname*{min}_{y\in F} \varphi _{y^0-D,k}(y)\}\\
 & = & \{y^0\in F\mid \varphi _{-D,k}(y^0-y^0)=\operatorname*{min}_{y\in F} \varphi _{-D,k}(y-y^0)\}.
\end{eqnarray*} 
\end{theorem}

\begin{proof}
For the relation between $\varphi _{-D,k}$ and $\varphi _{y^0-D,k}$, see Lemma \ref{A-shift}.\\
Consider some arbitrary $y^0\in F$. $y^0-D\subseteq \operatorname*{dom}\varphi _{y^0-D,k}$.\\
$y^0\in\operatorname*{WEff}(F,D)\Leftrightarrow F\cap (y^0-\operatorname*{int}D)\subseteq\{y^0\}\Leftrightarrow\varphi _{y^0-D,k}(y)\geq 0\mbox{ for all } y\in (F\cap\operatorname*{dom}\varphi _{y^0-D,k})\setminus\{y^0\}$.\\ 
$0\in D\Rightarrow y^0\in y^0-D.\Rightarrow \varphi _{y^0-D,k}(y^0)\leq 0$.\\
If $0\in\operatorname*{bd}D$, then $\varphi _{y^0-D,k}(y^0)=0$.
\end{proof}

\begin{remark}
The equation\\ 
$\operatorname*{WEff}(F,D) = \{y^0\in F\mid \varphi _{y^0-D,k}(y^0)=\operatorname*{min}_{y\in F} \varphi _{y^0-D,k}(y)\}$\\
was proved under a stronger condition which is equivalent to $k\in\operatorname*{int}0^+D$ for the case $0\in
D\setminus\operatorname*{int}D$ in \cite[Theorem 3.1]{GerWei90}. Under these assumptions, $\varphi _{y^0-D,k}$ is finite-valued.

For $Y=\mathbb{R}^\ell$ and $D=\mathbb{R} _{+}^{\ell}$, $\varphi _{y^0-\mathbb{R} _{+}^{\ell},k}$ had already been used by Bernau \cite{bern87} and by Brosowski \cite{bros87} for deriving scalarization results for the weakly efficient point set, where Bernau assumed $k\in \operatorname*{int}\mathbb{R}^{\ell}_+$ and Brosowski used $k=(1,\ldots ,1)^T$.
\end{remark}

Up to now, we have used functionals $\varphi _{y^0-D,k}$ for scalarizing the weakly efficient point set and the efficient point set, where $y^0$ was the (weakly) efficient element. We now turn to scalarization by functions $\varphi _{a-D,k}$, where $a$ is a fixed vector. In this case, $a$ can be a lower or an upper bound of $F$ and $D$ has to be a closed convex cone. In applications, an upper bound can easily be added to the vector minimization problem without any influence on the set of solutions. Note that scalarizations which are based on norms require a lower bound  of $F$. 

\begin{theorem}\label{hab-p432}
Suppose that $Y$ is a topological vector space and $D\subset Y$ a non-trivial closed convex cone with $\operatorname*{int}D\not=\emptyset$.\\
If $F\subseteq a-\operatorname*{int}D$ or $F\subseteq a+\operatorname*{int}D$ for some $a\in Y$, then
$$ \operatorname*{WEff}(F,D)=\{y^0\in F\mid \exists k\in\operatorname*{int}D:\; \varphi_{a-D,k}(y^0)=\min_{y\in F}\varphi_{a-D,k}(y)\} \quad\mbox{ and }$$
$$ \operatorname*{Eff}(F,D)=\{y^0\in F\mid \exists k\in\operatorname*{int}D\;\forall y\in F\setminus\{y^0\}:\; \varphi_{a-D,k}(y^0)<\varphi_{a-D,k}(y)\}.$$
For $y^0\in \operatorname*{WEff}(F,D)$, in the case $F\subseteq a-\operatorname*{int}D$ one can choose $k=a-y^0$, which results in 
$\varphi_{a-D,k}(y^0)=-1$, whereas in the case $F\subseteq a+\operatorname*{int}D$ one can choose $k=y^0-a$, which results in 
$\varphi_{a-D,k}(y^0)=1$.
\end{theorem}

\begin{proof}
Because of Corollary \ref{c-hab-t421}, we have only to show the inclusions $\subseteq$ of the equations.\\
Assume $F\subseteq a-\operatorname*{int}D$ and consider some $y^0\in \operatorname*{WEff}(F,D)$.\\
$\Rightarrow y^0\in F\subseteq a-\operatorname*{int}D$. $\Rightarrow k:=a-y^0\in \operatorname*{int}D$.\\
$\varphi_{a-D,k}(y)=-1\Leftrightarrow y\in a-\operatorname*{bd}D-k=y^0-\operatorname*{bd}D.$
Thus, $0\in \operatorname*{bd}D$ implies $\varphi_{a-D,k}(y^0)=-1$.\\
$\varphi_{a-D,k}(y)<-1\Leftrightarrow y\in a-\operatorname*{int}D-k=y^0-\operatorname*{int}D.$\\
$F\cap (y^0-\operatorname*{int}D)\subseteq\{y^0\}$ implies $\varphi_{a-D,k}(y)\not< -1\mbox{ for all } y\in F\setminus\{y^0\}$, thus
$\varphi_{a-D,k}(y^0)=\min_{y\in F}\varphi_{a-D,k}(y)$.\\
For $y^0\in \operatorname*{Eff}(F,D)$, $F\cap (y^0-D)=\{y^0\}$ and hence $\varphi_{a-D,k}(y)\not\leq -1\mbox{ for all } y\in F\setminus\{y^0\}$.\\
The case $F\subseteq a+\operatorname*{int}D$ can be handled in an analogous way. 
\end{proof}

\begin{remark}
The first equation in Theorem \ref{hab-p432} was also formulated by Luc \cite[p.95]{Luc89b} under the assumption $F\subseteq a-\operatorname*{int}D$.
\end{remark}

\begin{theorem}\label{hab-p433a}
Suppose that $Y$ is a topological vector space, $D\subset Y$ a non-trivial closed convex cone, that $F\subset Y$ contains more than one element and that
$$F\subseteq a-D$$
holds for some $a\in Y$. Then
$$ \operatorname*{Eff}(F,D)=\{y^0\in F\mid \exists k\in D\setminus (-D)\;\forall y\in F\setminus\{y^0\}:\; \varphi_{a-D,k}(y^0)<\varphi_{a-D,k}(y)\}.$$
For $y^0\in \operatorname*{Eff}(F,D)$, one can choose $k=a-y^0$, which results in\\ 
$\varphi_{a-D,k}(y^0)\leq -1<\varphi_{a-D,k}(y)\mbox{ for all } y\in F\setminus\{y^0\}$, where $\varphi_{a-D,k}$ is proper.
\end{theorem}

\begin{proof}
$F\subseteq a-D$ implies $F\subseteq\operatorname*{dom}\varphi _{a-D,k}$.\\
$\supseteq$ follows from Theorem \ref{hab-t421}(c) for $H=D$ since $\varphi _{a-D,k}$ is proper for $k\in D\setminus (-D)$ by Lemma \ref{c251a-C}.\\
Assume now $y^0\in \operatorname*{Eff}(F,D)$, i.e., $F\cap (y^0-D)=\{y^0\}$.\\
$F\not=\{a\} \Rightarrow y^0\not= a$. $\Rightarrow k:=a-y^0\in D\setminus\{0\}$.\\
$\varphi_{a-D,k}(y)\leq -1\Leftrightarrow y\in a-D-k=y^0-D.$\\
Hence, $\varphi_{a-D,k}(y^0)\leq -1<\varphi_{a-D,k}(y)\mbox{ for all } y\in F\setminus\{y^0\}$.\\
$\Rightarrow \varphi_{a-D,k}$ attains real values. $\varphi_{a-D,k}$ is proper by Lemma \ref{c251a-C}.
\end{proof}

\begin{theorem}\label{hab-p433b}
Suppose that $Y$ is a topological vector space and $D\subset Y$ a non-trivial pointed closed convex cone and $F\subset Y$ with
$$F\subseteq a+D$$
for some $a\in Y$. Then $\operatorname*{Eff}(F,D)=\{a\}$ in the case $a\in F$, and otherwise
\begin{eqnarray*}
\operatorname*{Eff}(F,D) &= &\{y^0\in F\mid \exists k\in D\setminus\{0\}:\;\;y^0\in\operatorname*{dom}\varphi _{a-D,k}
\mbox{  and }\\
& & \varphi_{a-D,k}(y^0)<\varphi_{a-D,k}(y) \mbox{ for all } y\in (F\cap\operatorname*{dom}\varphi _{a-D,k})\setminus\{y^0\}\}.
\end{eqnarray*}
For $y^0\in \operatorname*{Eff}(F,D)$, one can choose $k=y^0-a$, which results in\\ 
$\varphi_{a-D,k}(y^0)\leq 1<\varphi_{a-D,k}(y)\mbox{ for all } y\in (F\cap\operatorname*{dom}\varphi _{a-D,k})\setminus\{y^0\}$, where $\varphi_{a-D,k}$ is proper.
\end{theorem}

\begin{proof}
The case $F=\emptyset$ is trivial. Assume now $F\not=\emptyset$.\\
$\supseteq$ follows from Theorem \ref{hab-t421}(c) for $H=D$ since $\varphi _{a-D,k}$ is proper for $k\in D\setminus (-D)$ by Lemma \ref{c251a-C}.\\
Assume now $y^0\in \operatorname*{Eff}(F,D)$, i.e., $F\cap (y^0-D)=\{y^0\}$.\\
Obviously, $\operatorname*{Eff}(F,D)=\{a\}$ in the case $a\in F$. Suppose now $a\notin F$.\\
$\Rightarrow k:=y^0-a\in D\setminus\{0\}$. $\Rightarrow y^0=a-0+k\in a-D+\mathbb{R}k=\operatorname*{dom}\varphi_{a-D,k}$.\\
$\varphi_{a-D,k}(y)\leq 1\Leftrightarrow y\in a-D+k=y^0-D.$\\
Hence, $\varphi_{a-D,k}(y^0)\leq 1$ and $\varphi_{a-D,k}(y)>1\mbox{ for all } y\in (F\cap\operatorname*{dom}\varphi _{a-D,k})\setminus\{y^0\}$. $\varphi_{a-D,k}$ is proper by Lemma \ref{c251a-C} since $D$ is pointed.
\end{proof}

In Theorem \ref{hab-p433b}, the choice $k=y^0-a$ could result in a function $\varphi_{a-D,k}$ which does not attain real values if $D$ would not be pointed.

\begin{example}
Consider $Y=\mathbb{R}^2$, $F=\{(y_1,y_2)\in Y\mid 1\leq y_2\leq y_1+1\}$. $D:=\{(y_1,y_2)\in Y\mid y_1\geq 0\}$ is a closed convex cone. $F\subseteq a+D$ and $a\notin F$ for $a:=(0,0)^T$.  
$\operatorname*{Eff}(F,D)=(0,1)^T=:y^0$. For $k=y^0-a=(0,1)^T$, we have $k\in D\cap (-D)$. Hence, $\varphi_{a-D,k}$ does not attain real values.
\end{example}

Each weakly efficient point set of $F$ w.r.t. a convex cone $D$ that has a nonempty interior can be presented as the set of minimal solutions of a continuous strictly $(\operatorname*{int}D)$--monotone functional \cite[Satz 4.3.4]{Wei90}. In the convex case, the weakly efficient point set turns out to be the set of maximizers of a convex function.

\begin{proposition}
Assume that $Y$ is a topological vector space, $F\subset Y$, that $D\subset Y$ is a convex cone with $k\in \operatorname*{int}D$ and that $F+D$ is a closed proper subset of $Y$.
Then
\begin{eqnarray*}
\operatorname*{WEff}(F,D) & = & \{y^0\in F\mid \varphi_{F+D,-k}(y^0)=\operatorname*{max}_{y\in F}\varphi_{F+D,-k}(y)=0\}\\
& = & \{y^0\in F\mid \varphi_{B,k}(y^0)=\operatorname*{min}_{y\in F}\varphi_{B,k}(y)=0\}
\end{eqnarray*}
with $B:=Y\setminus \operatorname*{int}(F+D)$.\\
$\varphi_{F+D,-k}$ and $\varphi_{B,k}$ are finite-valued and continuous.\\
$\varphi_{B,k}$ is strictly $(\operatorname*{int}D)$-monotone, whereas $\varphi_{F+D,-k}$ is strictly $(-\operatorname*{int}D)$-monotone.\\
If $F+D$ is convex, then $\varphi_{F+D,-k}$ is convex and $\varphi_{B,k}$ is concave.
\end{proposition}

\begin{proof}
\begin{itemize}
\item[]
\item[(a)] $(H2_{F+D,-k})$ holds, thus $\varphi_{F+D,-k}$ is continuous on its domain.
Since $k\in\operatorname*{core}D$ and $D\subseteq 0^+(F+D)$, we get $k\in\operatorname*{core}0^+(F+D)$. 
Hence, $\varphi_{F+D,-k}$ is finite-valued by Lemma \ref{t251ua}.
Because of Lemma \ref{-phi_Ak}, $\varphi_{B,k}=-\varphi_{F+D,-k}$ is also finite-valued and continuous.
If $F+D$ is convex, then $\varphi_{F+D,-k}$ is convex by Lemma \ref{Fact4}, thus $\varphi_{B,k}$ is concave.
\item[(b)] Suppose $B-\operatorname*{int}D\subseteq B$ does not hold.
$\Rightarrow \exists y\in Y, d\in\operatorname*{int}D : y\notin \operatorname*{int}(F+D) \mbox{ and } y-d\in \operatorname*{int}(F+D)$. $\Rightarrow y\in d+\operatorname*{int}(F+D)\subseteq\operatorname*{int}(F+D)$, a contradiction.
Thus, $B-\operatorname*{int}D\subseteq B$, and $\varphi_{B,k}$ is strictly $(\operatorname*{int}D)$-monotone by
Lemma \ref{t251ua}. $\Rightarrow$ $\varphi_{F+D,-k}=-\varphi_{B,k}$ is strictly $(-\operatorname*{int}D)$-monotone.\\
By Proposition \ref{p-mon-wscal}, $\{y^0\in F\mid \varphi_{B,k}(y^0)=\operatorname*{min}_{y\in F}\varphi_{B,k}(y)\}\subseteq
\operatorname*{WEff}(F,D)$.\\
$\varphi_{F+D,-k}(y)\leq 0\mbox{ for all } y\in F+D$.\\
$\forall y\in Y:\;\varphi_{F+D,-k}(y)=0\Leftrightarrow y\in\operatorname*{bd}(F+D)$.\\
$\forall y\in Y:\;\varphi_{B,k}(y)=0\Leftrightarrow y\in\operatorname*{bd}B=\operatorname*{bd}(F+D)$.\\
$\Rightarrow\operatorname*{max}_{y\in F}\varphi_{F+D,-k}(y)=0=\operatorname*{min}_{y\in F}\varphi_{B,k}(y)$ since $\varphi_{F+D,-k}=-\varphi_{B,k}$.\\
Thus, $\{y^0\in F\mid \varphi_{F+D,-k}(y^0)=\operatorname*{max}_{y\in F}\varphi_{F+D,-k}(y)=0\}=\{y^0\in F\mid \varphi_{B,k}(y^0)=\operatorname*{Eff}_{y\in F}\varphi_{B,k}(y)=0\}\subseteq
\operatorname*{WEff}(F,D)$.
\item[(c)] By Theorem \ref{weff-in-bd}, $\operatorname*{WEff}(F,D)\subseteq\operatorname*{bd}(F+D)$. Thus, $\varphi_{F+D,-k}(y)=0\mbox{ for all } y\in\operatorname*{WEff}(F,D)$.
This yields the assertion.
\end{itemize}
\end{proof}
\smallskip

\section{Scalarization of the Efficient and the Weakly Efficient Point Set by Norms}\label{s-scal-norm}

Section \ref{s-scal-uni} delivers by Lemma \ref{p-ordint-varphi} scalarization results for efficient and weakly efficient elements by norms. 

We will use the following condition:

\quad

\begin{tabular}{ll}
(Hn--VOP$_{C,k}$): & $Y$ is a topological vector space, $D$ is a nonempty subset of $Y$,\\
& $C$ is a non-trivial closed convex pointed cone in $Y$ with\\
&  $k\in \operatorname*{core}C$, $a\in Y$ and $F\subset Y$ with $F\subseteq a+C$. 
\end{tabular}

\quad

In this section, $\|\cdot \| _{C,k}$ will denote the norm which is given as the Minkowski functional of the order interval $[-k,k]_C$ (see \cite[Lemma 1.45]{Jah86a}). This norm is just an order unit norm. 

Let us first give some sufficient conditions for efficient and weakly efficient points by minimal solutions of norms. 
Theorem \ref{hab-t421} implies with Lemma \ref{p-ordint-varphi}:

\begin{theorem}\label{n-hab-t421}
Suppose (Hn--VOP$_{C,k}$) and define
$$\Psi := \operatorname*{argmin}_{y\in F} \| y-a\| _{C,k}.$$ 
Then:
\begin{itemize}
\item[(a)] $\operatorname*{Eff}(F,D)\cap \Psi\subseteq \operatorname*{Eff}(\Psi,D)$.
\item[(b)] $C+D\subseteq C\implies  \operatorname*{Eff}(F,D)\cap \Psi= \operatorname*{Eff}(\Psi,D)$.
\item[(c)] $C+D\subseteq C$ and $\Psi=\{ y^0\}$ imply $y^0\in\operatorname*{Eff}(F,D)$.
\item[(d)] $C+(D\setminus\{0\})\subseteq \operatorname*{core}C \implies\Psi\subseteq \operatorname*{Eff}(F,D)$.
\item[(e)] $C+\operatorname*{int}D\subseteq C \implies  \Psi\subseteq \operatorname*{WEff}(F,D)$.
\end{itemize}
\end{theorem}

\begin{corollary}\label{n-c-hab-t421}
Suppose (Hn--VOP$_{D,k}$) and define
$$\Psi := \operatorname*{argmin}_{y\in F} \| y-a\| _{D,k}.$$
Then:
\begin{itemize}
\item[(a)] $\Psi\subseteq \operatorname*{WEff}(F,D)$.
\item[(b)] $\operatorname*{Eff}(F,D)\cap \Psi= \operatorname*{Eff}(\Psi,D)$.
\item[(c)] $\Psi=\{ y^0\}$ implies $y^0\in\operatorname*{Eff}(F,D)$.
\end{itemize}
\end{corollary}

We will now characterize the efficient point set and the weakly efficient point set by minimal solutions of norms. 
We get from Theorem \ref{hab-p432} by Lemma \ref{p-ordint-varphi}:

\begin{theorem}\label{n-hab-p432}
Suppose that $Y$ is a topological vector space and $D\subset Y$ a non-trivial closed convex pointed cone with $\operatorname*{int}D\not=\emptyset$.\\
If $F\subseteq a+\operatorname*{int}D$ for some $a\in Y$, then
$$ \operatorname*{WEff}(F,D)=\{y^0\in F\mid \exists k\in\operatorname*{int}D:\; \| y^0-a\| _{D,k}=\operatorname*{min}_{y\in F}\| y-a\| _{D,k}\} \quad\mbox{ and }$$
$$ \operatorname*{Eff}(F,D)=\{y^0\in F\mid \exists k\in\operatorname*{int}D\;\forall y\in F\setminus\{y^0\}:\; \| y^0-a\| _{D,k}<\| y-a\| _{D,k}\}.$$
For $y^0\in \operatorname*{WEff}(F,D)$, one can choose $k=y^0-a$, which results in 
$\| y^0-a\| _{D,k}=1$.
\end{theorem}

\smallskip

%\bibliographystyle{spmpsci0}
%\bibliography{ct,cz,c_zali,wd,wd_vekopt}
\def\cfac#1{\ifmmode\setbox7\hbox{$\accent"5E#1$}\else
  \setbox7\hbox{\accent"5E#1}\penalty 10000\relax\fi\raise 1\ht7
  \hbox{\lower1.15ex\hbox to 1\wd7{\hss\accent"13\hss}}\penalty 10000
  \hskip-1\wd7\penalty 10000\box7}
  \def\cfac#1{\ifmmode\setbox7\hbox{$\accent"5E#1$}\else
  \setbox7\hbox{\accent"5E#1}\penalty 10000\relax\fi\raise 1\ht7
  \hbox{\lower1.15ex\hbox to 1\wd7{\hss\accent"13\hss}}\penalty 10000
  \hskip-1\wd7\penalty 10000\box7}
  \def\cfac#1{\ifmmode\setbox7\hbox{$\accent"5E#1$}\else
  \setbox7\hbox{\accent"5E#1}\penalty 10000\relax\fi\raise 1\ht7
  \hbox{\lower1.15ex\hbox to 1\wd7{\hss\accent"13\hss}}\penalty 10000
  \hskip-1\wd7\penalty 10000\box7}
  \def\cfac#1{\ifmmode\setbox7\hbox{$\accent"5E#1$}\else
  \setbox7\hbox{\accent"5E#1}\penalty 10000\relax\fi\raise 1\ht7
  \hbox{\lower1.15ex\hbox to 1\wd7{\hss\accent"13\hss}}\penalty 10000
  \hskip-1\wd7\penalty 10000\box7}
  \def\cfac#1{\ifmmode\setbox7\hbox{$\accent"5E#1$}\else
  \setbox7\hbox{\accent"5E#1}\penalty 10000\relax\fi\raise 1\ht7
  \hbox{\lower1.15ex\hbox to 1\wd7{\hss\accent"13\hss}}\penalty 10000
  \hskip-1\wd7\penalty 10000\box7}
  \def\cfac#1{\ifmmode\setbox7\hbox{$\accent"5E#1$}\else
  \setbox7\hbox{\accent"5E#1}\penalty 10000\relax\fi\raise 1\ht7
  \hbox{\lower1.15ex\hbox to 1\wd7{\hss\accent"13\hss}}\penalty 10000
  \hskip-1\wd7\penalty 10000\box7}
  \def\cfac#1{\ifmmode\setbox7\hbox{$\accent"5E#1$}\else
  \setbox7\hbox{\accent"5E#1}\penalty 10000\relax\fi\raise 1\ht7
  \hbox{\lower1.15ex\hbox to 1\wd7{\hss\accent"13\hss}}\penalty 10000
  \hskip-1\wd7\penalty 10000\box7}
  \def\cfac#1{\ifmmode\setbox7\hbox{$\accent"5E#1$}\else
  \setbox7\hbox{\accent"5E#1}\penalty 10000\relax\fi\raise 1\ht7
  \hbox{\lower1.15ex\hbox to 1\wd7{\hss\accent"13\hss}}\penalty 10000
  \hskip-1\wd7\penalty 10000\box7}
  \def\cfac#1{\ifmmode\setbox7\hbox{$\accent"5E#1$}\else
  \setbox7\hbox{\accent"5E#1}\penalty 10000\relax\fi\raise 1\ht7
  \hbox{\lower1.15ex\hbox to 1\wd7{\hss\accent"13\hss}}\penalty 10000
  \hskip-1\wd7\penalty 10000\box7} \def\Dbar{\leavevmode\lower.6ex\hbox to
  0pt{\hskip-.23ex \accent"16\hss}D}
  \def\cfac#1{\ifmmode\setbox7\hbox{$\accent"5E#1$}\else
  \setbox7\hbox{\accent"5E#1}\penalty 10000\relax\fi\raise 1\ht7
  \hbox{\lower1.15ex\hbox to 1\wd7{\hss\accent"13\hss}}\penalty 10000
  \hskip-1\wd7\penalty 10000\box7} \def\cprime{$'$}
  \def\Dbar{\leavevmode\lower.6ex\hbox to 0pt{\hskip-.23ex \accent"16\hss}D}
  \def\cfac#1{\ifmmode\setbox7\hbox{$\accent"5E#1$}\else
  \setbox7\hbox{\accent"5E#1}\penalty 10000\relax\fi\raise 1\ht7
  \hbox{\lower1.15ex\hbox to 1\wd7{\hss\accent"13\hss}}\penalty 10000
  \hskip-1\wd7\penalty 10000\box7} \def\cprime{$'$}
  \def\Dbar{\leavevmode\lower.6ex\hbox to 0pt{\hskip-.23ex \accent"16\hss}D}
  \def\cfac#1{\ifmmode\setbox7\hbox{$\accent"5E#1$}\else
  \setbox7\hbox{\accent"5E#1}\penalty 10000\relax\fi\raise 1\ht7
  \hbox{\lower1.15ex\hbox to 1\wd7{\hss\accent"13\hss}}\penalty 10000
  \hskip-1\wd7\penalty 10000\box7}
  \def\udot#1{\ifmmode\oalign{$#1$\crcr\hidewidth.\hidewidth
  }\else\oalign{#1\crcr\hidewidth.\hidewidth}\fi}
  \def\cfac#1{\ifmmode\setbox7\hbox{$\accent"5E#1$}\else
  \setbox7\hbox{\accent"5E#1}\penalty 10000\relax\fi\raise 1\ht7
  \hbox{\lower1.15ex\hbox to 1\wd7{\hss\accent"13\hss}}\penalty 10000
  \hskip-1\wd7\penalty 10000\box7} \def\Dbar{\leavevmode\lower.6ex\hbox to
  0pt{\hskip-.23ex \accent"16\hss}D}
  \def\cfac#1{\ifmmode\setbox7\hbox{$\accent"5E#1$}\else
  \setbox7\hbox{\accent"5E#1}\penalty 10000\relax\fi\raise 1\ht7
  \hbox{\lower1.15ex\hbox to 1\wd7{\hss\accent"13\hss}}\penalty 10000
  \hskip-1\wd7\penalty 10000\box7} \def\Dbar{\leavevmode\lower.6ex\hbox to
  0pt{\hskip-.23ex \accent"16\hss}D}
  \def\cfac#1{\ifmmode\setbox7\hbox{$\accent"5E#1$}\else
  \setbox7\hbox{\accent"5E#1}\penalty 10000\relax\fi\raise 1\ht7
  \hbox{\lower1.15ex\hbox to 1\wd7{\hss\accent"13\hss}}\penalty 10000
  \hskip-1\wd7\penalty 10000\box7} \def\Dbar{\leavevmode\lower.6ex\hbox to
  0pt{\hskip-.23ex \accent"16\hss}D}
  \def\cfac#1{\ifmmode\setbox7\hbox{$\accent"5E#1$}\else
  \setbox7\hbox{\accent"5E#1}\penalty 10000\relax\fi\raise 1\ht7
  \hbox{\lower1.15ex\hbox to 1\wd7{\hss\accent"13\hss}}\penalty 10000
  \hskip-1\wd7\penalty 10000\box7}

\end{document}